\documentclass[twoside,11pt]{article}

\usepackage[english]{babel}
\usepackage{amsmath}
\usepackage{amsthm}
\usepackage{cases}
\usepackage{amsfonts}
\usepackage{amssymb}
\usepackage{graphicx}
\usepackage{colortbl,dcolumn}
\usepackage{float}
\usepackage{paralist}  
\allowdisplaybreaks[4]

\newtheorem{lemma}{\bf Lemma}[section]
\newtheorem{theorem}{\bf Theorem}[section]
\newtheorem{proposition}{\bf Proposition}[section]

\newtheorem{remark}{\bf Remark}[section]

\numberwithin{equation}{section}

\begin{document}
	
	\title{{\Large Traveling waves to a logarithmic chemotaxis model with fast diffusion and singularities}
		\footnotetext{\small
			*Corresponding author.}
		\footnotetext{\small E-mail addresses: lixw928@nenu.edu.cn (XL), dfli@hust.edu.cn (DL), lijy645@nenu.edu.cn (JL), ming.mei@mcgill.ca (MM).} }
	
	\author{{Xiaowen Li$^{1,4}$, Dongfang Li$^2$, Jingyu Li$^{1,}$$^\ast$ and Ming Mei$^{3,4}$}\\[2mm]
		\small\it $^1$School of Mathematics and Statistics, Northeast Normal University,\\
		\small\it   Changchun, 130024, P.R.China \\
        \small\it $^2$School of Mathematics and Statistics, Huazhong University of Science and Technology, \\
        \small\it Wuhan, 430074,  P.R.China \\
		\small\it $^3$Department of Mathematics, Champlain College Saint-Lambert,\\
		\small\it     Saint-Lambert, Quebec, J4P 3P2, Canada\\
		\small\it $^4$Department of Mathematics and Statistics, McGill University,\\
		\small\it     Montreal, Quebec, H3A 2K6, Canada  }

	\date{}
	
	\maketitle
	
	\begin{quote}
		\small \textbf{Abstract}: This paper is concerned with a chemotaxis model with logarithmic sensitivity and fast diffusion, which possesses strong singularities for the sensitivity at zero-concentration of chemical signal, and for the diffusion at zero-population of cells, respectively. The main purpose is to show the existence of traveling waves connecting the singular zero-end-state, and particularly,  to show the asymptotic stability of these traveling waves. The challenge of the problem is the interaction of two kinds of singularities involved in the model: one is the logarithmic singularity of the sensitivity; and the other is the power-law singularity of the diffusivity. To overcome the singularities for the wave stability, some new techniques of weighted energy method are introduced artfully. Numerical simulations are also carried out, which further confirm our theoretical stability results, in particular, the numerical results indicate that the effect of fast diffusion to the structure of traveling waves is essential, which causes the traveling waves much steeper like shock waves. This new phenomenon is a first observation.
		
		
\indent \textbf{AMS (2020) Subject Classification}: 35C07, 35B35, 35Q92, 92C17

		
	\end{quote}

	
\section{Introduction}

{\bf Models and Background}. In this paper, we study the chemotaxis model \cite{typicalK-S,Background,OFSF,VP}
\begin{equation}\label{orginal model}
	\left\{\begin{aligned}
		u_t&=(u^p)_{x x}-\chi \left[u(\ln w)_x\right]_x, \\
		w_t&=-u w,
	\end{aligned}\right.
\end{equation}
where $u(x,t)$ is the population density of cells or bacteria, and $w(x,t)$ is the concentration of chemical signal (e.g. nutrient). The cells diffuse nonlinearly in the form of $(u^p)_{x x}$, where $p>0$ denotes the diffusion exponent, which means the diffusion of cells is fast if $0<p<1$, linear if $p=1$ or slow if $p>1$. This nonlinear diffusion models the movement of cells in a porous medium or the mechanism of prevention of overcrowding for cells. The chemotactic response of the cells to the stimulation of the chemical signal is represented by $-\chi \left[u(\ln w)_x\right]_x$, where $\chi>0$ represents the chemotactic coefficient that measures the strength of chemoattractants. Here the logarithmic sensitivity function $\varphi(w)=\ln w$ comes from the pervasiveness of Weber-Fechner law, and was first introduced by Keller and Segel \cite{typicalK-S} to model the propagation of traveling band of bacteria observed in the celebrated experiment of Adler \cite{Adler2}. In the second equation, we neglect the diffusion of chemical signal because in many circumstances (e.g. Adler's experiment \cite{Adler2}) the diffusion rate of typical signal is quite small compared to that of the cells. This model possesses two singularities: one is the singularity of sensitivity $\ln w$ when $w=0$, and the other is the singularity of fast diffusion $(u^p)_{xx}=p(\frac{u_x}{u^{1-p}})_x$ for $0<p<1$ when $u=0$.

When $p=1$, the system \eqref{orginal model} was proposed by Levine \emph{et al.} \cite{Background} to model the initiation of tumor angiogenesis, where $u$ and $w$ denote the density of vascular endothelial cells and the concentration of vascular endothelial growth factor, respectively. In this biological process, the diffusion of endothelial growth factor is neglected since it is less important than its interaction with the endothelial cells. The mathematical derivation of this singular chemotaxis model has been obtained in Refs.~\cite{hopf-cole} and \cite{Othmer} from the viewpoint of reinforced random walks. We also refer to Refs.~\cite{Bellomo15} and \cite{Bellomo} for the derivations of a
variety of cross-diffusion models via the micro-macro asymptotic method.

Since the pioneering work of Keller and Segel \cite{typicalK-S}, the analysis of traveling waves has become one of the most important research topics in chemotaxis, as many complex wave patterns have been observed in various experiments of chemotaxis \cite{Adler2,Gold,Narla,Saragosti,Welch}. For instance, E. coli in a capillary tube might move as a traveling band by consuming oxygen \cite{Adler2}, Dictyostelium discoideum generated a spiral wave in seeking cyclic adenosine monophosphate \cite{Gold}, and myxobacteria migrated towards some chemical signal as a periodic traveling wave \cite{Welch}. In mathematics, most of the achievements are concerned with the case of linear diffusion (i.e.~$p=1$). See Ref.~\cite{Hillen} for the first result on the existence of monotone traveling wave, Refs. \cite{HJin,20JDE} for the stability of such traveling waves under appropriate perturbations, and Refs.~\cite{Henderson,Li-Wang,Salako} for the effect of logistic growth. When the chemical diffusion is considered in the second equation, the existence and stability of monotone traveling wave solutions were studied in Ref.~\cite{stabilitydiffusion}. For more analytical works on traveling waves of chemotaxis models with linear diffusion, we refer the interested readers to Ref.~\cite{Wangreview}.

When $p\neq1$, the system \eqref{orginal model} models the chemotactic behaviors of cells or bacteria in porous media. The chemotaxis of cells and bacteria in porous media is important in both experiments and mathematics. Indeed, experiments were took in Refs.~\cite{OFSF} and \cite{VP} to quantify the bacterial chemotaxis in porous media. To prevent overcrowding of bacteria, nonlinear diffusion for chemotaxis models was introduced in Refs.~\cite{BFD,Carrillo,Choi-Wang,HP}. Furthermore, Mendez \textit{et al.}~\cite{men} derived mathematically the chemotaxis model with nonlinear diffusion from the viewpoint of Markovian reaction-random walks. There are also fruitful theoretical works on the chemotaxis model of self-aggregation type with nonlinear diffusion, where the chemical signal is generated by the cells. See Refs.~\cite{Bian,Carrillo-22,Sugiyama-Kawakami,Sugiyama-Yahagi} for the uniqueness, extinction behavior, blowup problem and stationary distributions of the solutions if the diffusion of cells is fast, Refs.~\cite{Tao-Winkler12,Tao-Winkler13} for the global boundedness of solutions, and Refs.~\cite{Bedrossian,Bian,Blanchet,Chen} for the colorful dynamics of the system if the diffusion of cells is slow. We refer to Ref.~\cite{Carrillo-19} for a systematical review of literatures in this direction.

If the chemotaxis model with nonlinear diffusion is of consumption type, like \eqref{orginal model}, the understanding of its dynamics is quite limited. For the slow diffusion case, Yan and Li \cite{Yan} constructed a global generalized solution to the system \eqref{orginal model} on a bounded domain with Neumann boundary condition, and Jin \cite{Jin} showed the existence of global weak solutions for large initial data and proved the stability of constant steady states to a chemotaxis-consumption model with linear sensitivity. Regarding the traveling waves, Kim and his collaborators \cite{CK15,CK20} are the first to construct compactly supported traveling waves to \eqref{orginal model} in the slow diffusion region (i.e.~$p>1$), while the problem of stability of such waves was left open due to the strong degeneracy of slow diffusion. Very recently, Refs.~\cite{Arias18,Campos23} have constructed a series of exotic traveling waves that may be regular, singular and even discontinuous for the porous-media
flux-saturated chemotaxis models, while the stability of these traveling waves are unknown. The purpose of this paper is to deal with the much less understood case of fast diffusion (i.e.~$0<p<1$). We shall show both the existence and stability of traveling waves in the fast diffusion case.

{\bf Traveling waves and singularities}. Before introducing the challenges of this problem, we present some elementary properties of the traveling waves. Here we mean the traveling wave of system \eqref{orginal model} is a non-constant smooth self-similar solution $(U, W)$ in the form of
$$
(u, w)(x, t)=(U, W)(z), \quad z=x-s t,\quad x\in(-\infty,\infty),
$$
satisfying
\begin{equation}\label{oriTW boundary data}
	U( \pm \infty)=u_{ \pm}, \quad W( \pm \infty)=w_{ \pm}, \quad U^{\prime}( \pm \infty)=W^{\prime}( \pm \infty)=0,
\end{equation}
where $'=\frac{d}{dz}$, $z$ is called the moving coordinate and $s>0$ is the wave speed. Substituting this wave ansatz into the system \eqref{orginal model} yields the equations of $(U,W)$
\begin{equation}\label{oritravleing wave equ}\begin{cases}
		-sU_z=(U^p)_{zz}-\chi\left[U (\ln W)_z \right]_z, \\
		s W_z=UW.
	\end{cases}
\end{equation}
Notice that by \eqref{oriTW boundary data} and the second equation of \eqref{oritravleing wave equ}, it holds that
\begin{equation}\label{upmwpm}
	u_{\pm}w_{\pm}=0.
\end{equation}
On the other hand, thanks to $u_{\pm}\geq 0$, $w_{\pm}\geq0$ and $s>0$, we get from the second equation of \eqref{oritravleing wave equ} that $W_z>0$. This implies that $w_+>0$. Hence, by \eqref{upmwpm}, we have $u_+=0$ and subsequently $u_->0$. Utilizing \eqref{upmwpm} again, it gives $w_-=0$. Then the following properties of $(U,W)(z)$  can be verified:
\begin{equation}\label{properties}
	U_z<0,~ W_z>0,~ \text { for } z \in(-\infty,+\infty),~ w_{+}>0,~ u_{-}>0,~ u_{+}=w_{-}=0.
\end{equation}
Therefore, to investigate the stability of traveling waves to \eqref{orginal model}, it is necessary to equip \eqref{orginal model} with the following initial conditions
\begin{equation}\label{initial data}
	(u, w)(x, 0)=\left(u_0, w_0\right)(x) \rightarrow\left\{\begin{aligned}
		(0,w_+),&\text { as } x \rightarrow +\infty, \\
		(u_-,0),&\text { as } x \rightarrow -\infty.
	\end{aligned}\right.
\end{equation}

One can observe from the fact $u_+=w_-=0$ that, the system \eqref{orginal model} with $0<p<1$ has two types of singularities: one is the singular sensitivity at the constant-state $w_-=0$, and the other is the singular diffusivity at the constant-state $u_+=0$.

In this paper, we shall exploit some novel strategies to overcome the analytical difficulties generated by the interaction of the two singularities. As in the arguments of Refs.~\cite{HJin,20JDE,Nonlinearstability-1}, we take the following Hopf-Cole transformation
\begin{equation}\label{hopf-cole}
	v=-(\ln w)_x,
\end{equation}
to remove the singular logarithmic sensitivity. Then we transfer the system \eqref{orginal model} with \eqref{initial data} into the following parabolic-hyperbolic system
\begin{equation}\label{transformed model}
	\left\{\begin{aligned}
		u_t&=(u^p)_{x x}+\chi(u v)_x, \\
		v_t&=u_x,
	\end{aligned}\right.
\end{equation}
with initial data
\begin{equation}\label{transinitial data}
	(u, v)(x, 0)=\left(u_0, v_0\right)(x) \rightarrow\left\{\begin{aligned}
		&(0,0),\ \text { as } x \rightarrow +\infty, \\
		&(u_-,v_-),\ \text { as } x \rightarrow -\infty,
	\end{aligned}\right.
\end{equation}
where
\begin{equation*}
	v_0(x)=-(\ln w_0)_x \text{ and }  v_{-}=-\lim_{x \rightarrow-\infty} \frac{w_{0 x}}{w_0}.
\end{equation*}
Now the primary obstacle is the singular diffusivity since $(u^p)_{xx}=p(u^{p-1}u_x)_x$. In addition, the nonlinear diffusion $(u^p)_{xx}$ results in undesirable tedious estimates. In this paper, we shall develop some new ideas to establish the existence and stability of traveling waves to the system \eqref{orginal model} with $0<p<1$. Specifically, we are going to show that:
\begin{enumerate}[(i)]
	\item the  system \eqref{orginal model} with \eqref{initial data} admits a unique (up to a shift) monotone traveling wave solution $(U, W)(x-st)$ satisfying
	$U_z(z)<0$ and $W_z(z)>0$ for all $z\in \mathbb{R}$ (see Theorem \ref{original TW existence});
	
	\item the traveling wave $(U, W)$ obtained above is nonlinearly asymptotically stable in some topological space (see Theorem \ref{original TW stability}).
\end{enumerate}

The stability of traveling waves with fast diffusions and the singular state $u_+=0$ is also numerically demonstrated in different cases. Remarkably, the numerical simulations show that the stronger fast diffusion causes the shape of traveling waves to be more steeper like shock waves, and the oscillations slowly disappear. This means that the effect of fast diffusion with the singular state $u_+=0$ is essential for the structure of traveling waves. This new phenomenon is a first observation.

{\bf Strategies for treating singularities}. To prove the result (i), we first establish the existence of monotone traveling wave solution $(U,V)$ to transformed system \eqref{transformed model}-\eqref{transinitial data}. Based on the conservation structure of this system, it is easy to see that $(U,V)$ satisfies a first-order ordinary differential equation that can be easily resolved. Then we define the wave component $W$ through the second equation of \eqref{oritravleing wave equ}. Therefore, the existence of traveling waves to the system \eqref{orginal model} can be obtained without too much analytical effort. In order to show the result (ii), it is natural to study the asymptotic stability of $(U,V)$ to the system \eqref{transformed model}-\eqref{transinitial data} first. However, we still need to deal with the challenge of the singularity of  $(u^{p-1}u_x)_x$ since $u_+=0$ and $0<p<1$. To settle this difficulty, we first employ the technique of taking anti-derivatives to classify the strength of the singularity. Then according to the structure of the fast diffusion, we take a weighted functional space to control the singularity where the weights are carefully constructed. In particular, we have to select a piecewise smooth weight function in the basic $L^2$ estimate. Furthermore, in the $H^1$ and $H^2$ estimates, the interaction of the singular diffusivity and weighted functional space results in an unfavorable obstacle that, non-zero boundary terms may arise in the integration by parts. We break down this obstacle by introducing a regular approximate weight function that removes part of the singularity. After establishing uniform estimates that are independent of the artificial parameter, we obtain the desired energy estimates by using the Fatou's Lemma. To our knowledge, this is the first analytical work on the nonlinear stability
of traveling waves to the partial differential equations (PDE) with fast diffusion. The approach based on weighted energy estimates is flexible and is expected to be useful for the study of stability of wave patterns to other more sophisticated PDE models.

Before concluding this section, we recall some related works on the transformed system \eqref{transformed model} with $p=1$. In the one-dimensional case, the nonlinear stability of traveling waves was first obtained by Li and Wang \cite{Nonlinearstability-1} under appropriate perturbations if the waves are away from zero. Especially, the authors discovered a cancellation structure in the energy estimate that was quite helpful in the following studies. Recently, Choi \textit{et al.}~\cite{Choi} proved the orbital stability of traveling waves to \eqref{transformed model} with $p=1$ under general perturbations. The global existence of large solutions for the  Cauchy problem and the initial-boundary value problem were carried out in Refs.~\cite{globalwellposedness3,DLiPanZhao} and \cite{Martinez,globalwellposedness2}, respectively. In the multidimensional case, the global existence of small solutions to the Cauchy problem and the initial-boundary value problem were studied in Refs.~\cite{multidimensional2,multidimensional-3} and \cite{LiPanZhao,Rebholz}, respectively. When the chemical diffusion is included in the second equation, we refer to Refs.~\cite{Hou,Wang-Xiang,Winkler-2016,Winkler-2018} for various interesting works.

The rest of this paper is organized as follows. In section \ref{main results}, we state our main results on the existence and nonlinear stability of traveling wave solutions to the parabolic-hyperbolic system \eqref{transformed model} and the original chemotaxis system \eqref{orginal model}, respectively. Section \ref{existence} is devoted to the proof of the existence of traveling waves. In section \ref{nonlinear stability}, we first show the nonlinear stability of traveling waves to the system \eqref{transformed model} on the basis of weighted energy estimates, and then transfer the stability result back to the original chemotaxis model \eqref{orginal model}. In section \ref{numerical}, we present some numerical simulations in different cases, which further perfectly confirm our theoretical results of wave stability.

\section{Preliminaries and main results}\label{main results}

By using the Hopf-Cole transformation \eqref{hopf-cole}, the chemotaxis system \eqref{orginal model} is transformed into the parabolic-hyperbolic system \eqref{transformed model}, in which the logarithmic singularity of the sensitive function is removed. In the following we shall first study the existence and nonlinear stability of traveling wave solutions to the transformed system \eqref{transformed model}-\eqref{transinitial data}, and then transfer the results back to the original chemotaxis system.

The traveling wave solution of \eqref{transformed model}-\eqref{transinitial data} is a non-constant self-similar solution given by
\begin{equation}
	u(x,t)=U(z), \	v(x,t)=V(z),\quad z=x-st.
\end{equation}
Substituting this wave ansatz into \eqref{transformed model}, we get the following equations of $(U,V)$
\begin{equation}\label{travleing wave equ}\begin{cases}
		-sU_z=(U^p)_{zz}+\chi(U V)_z, \\
		-s V_z=U_z.
	\end{cases}
\end{equation}
From \eqref{properties} and \eqref{hopf-cole}, one can easily see that
\begin{equation}\nonumber
	V(z)=-\frac{W_z(z)}{W(z)}<0,\quad V(+\infty)=-\lim _{z \rightarrow+\infty} \frac{W_z(z)}{W(z)}=0 .
\end{equation}
Hence, the boundary conditions of \eqref{travleing wave equ} read
\begin{equation}\label{TW boundary data}
	\!U(\!+\!\infty)\!=\!V(\!+\!\infty)\!=0,~U(\!-\!\infty)\!\!=u_->0,~V(\!-\!\infty\!)=\!v_-<0,~ U^{\prime}( \pm \infty\!)=\!V^{\prime}( \pm \infty)\!=\!0.
\end{equation}

It is expected that $\underset{z\rightarrow\pm\infty}{\lim}(U^p)_z(z)=0$. Then integrating \eqref{travleing wave equ} over $(-\infty,+\infty)$ yields the Rankine-Hugoniot condition
\begin{equation}\label{R-H condition}
	\begin{cases}
		-s\left(0-u_{-}\right)-\chi\left(0-u_{-} v_{-}\right)=0, \\
		-s\left(0-v_{-}\right)-\left(0-u_{-}\right)=0,
	\end{cases}
\end{equation}
which further gives
\begin{equation}\label{s^2}
	s=\sqrt{\chi u_-} \ \text{ and } \	v_-=-\frac{u_-}{s}<0.
\end{equation}

\begin{theorem}\label{existence of TW}
	Suppose that $(u_-,v_-)$ satisfies \eqref{R-H condition}. Then the system \eqref{travleing wave equ}-\eqref{TW boundary data} has a unique (up to a translation) smooth monotone traveling wave solution $(U, V)(z)$ satisfying
	\begin{equation}\label{Uz}
		U_z(z)<0,\quad V_z(z)>0,\quad \forall \ z\in \mathbb{R},
	\end{equation} and the wave speed $s=-\chi v_->0$.
	Moreover,
	\begin{equation}\begin{split}
			U(z)&\sim Cz^{-\frac{1}{1-p}}, \ \text{ as } z\rightarrow +\infty,\\
			u_--U(z)&\sim C\mathrm{e}^{\frac{\chi}{sp}u_{-}^{2-p}z}, \ \text{ as } z\rightarrow -\infty.\end{split}\end{equation}
	
\end{theorem}

\begin{remark} The wave component $U$ decays at an algebraic rate as $z\rightarrow+\infty$. This is in contrast to the case of linear diffusion where the waves decay to $0$ at an exponential rate as $z\rightarrow+\infty$ (see Refs. \cite{HJin} and \cite{20JDE}). This phenomenon also verifies the principle that the range $0<p<1$ is usually referred to as \lq\lq fast diffusion\rq\rq, in the sense that the faster diffusion (than linear diffusion) occurs for small values of the density.
	
\end{remark}

We next investigate the stability of the traveling wave solution $(U,V)(x-st)$ to the transformed system \eqref{transformed model}-\eqref{transinitial data}. One may observe that this new system still has a strong singularity in the diffusivity since $0<p<1$, which leads to a challenging problem. In this paper, we shall develop some novel ideas to resolve this strong singularity.

To state our result more precisely, we introduce some notation. The integrals $\int_{\mathbb{R}}f(x)\mathrm{d}x$ and $\int_{0}^{t}\int_{\mathbb{R}}f(x,\tau)\mathrm{d}x\mathrm{d}\tau$ will be abbreviated as $\int f(x)$ and $\int_{0}^{t}\int f(x,\tau)$, respectively. $H^{k}(\mathbb{R})$ is the usual $k$-th order Sobolev space with norm $\|f\|_{H^k(\mathbb{R})}:=\left(\sum_{j=0}^k\|\partial_x^j f\|_{L^2(\mathbb{R})}^2\right)^{\frac{1}{2}}$. $H^{k}_{w}(\mathbb{R})$ denotes the weighted space of measurable functions $f$ so that $\sqrt{w} \partial_x^j f \in L^2$ for $0\leq j\leq k$ with norm $\|f\|_{H_w^k(\mathbb{R})}:=\left(\sum_{j=0}^k \int w(x)|\partial_x^j f|^2 d x\right)^{\frac{1}{2}}$. For simplicity, we denote $\|\cdot\|:=\|\cdot\|_{L^2(\mathbb{R})}$, $\|\cdot\|_w:=\|\cdot\|_{L_w^2(\mathbb{R})}$, $\|\cdot\|_k:=\|\cdot\|_{H^k(\mathbb{R})}$ and $\|\cdot\|_{k, w}:=\|\cdot\|_{H_w^k(\mathbb{R})}$.

The nonlinear stability of the traveling wave solution to transformed system \eqref{transformed model}-\eqref{transinitial data} is stated as follows.

\begin{theorem}\label{transformed stability}
	Let $(U,V)(x-st)$ be the traveling wave solution obtained in Theorem \ref{existence of TW}. Assume that there exists a constant $x_0$ such that $\phi_0(+\infty)=\psi_0(+\infty)=0$ where
	\begin{equation*}
		\left(\phi_0, \psi_0\right)(x):=\int_{-\infty}^x\left(u_0(y)-U(y+x_0), v_0(y)-V(y+x_0)\right) dy.
	\end{equation*}
	Then there exists a positive constant $\epsilon_{0}$ such that if $\left\|u_0-U\right\|_{1, w_3}+\left\|v_0-V\right\|_{1, w_4}+\left\|\phi_0\right\|_{w_1}+\left\|\psi_0\right\|_{w_2} \leq \epsilon_0$, then the Cauchy problem \eqref{transformed model}-\eqref{transinitial data} admits a unique global solution $(u,v)(x,t)$ satisfying
	\begin{equation}\label{2.10}
		u-U \in C\left([0, \infty) ; H_{w_3}^1\right),\
		v-V \in C\left([0, \infty) ; H_{w_4}^1\right),	
	\end{equation}
	where the weight functions $w_i$, $i=1,\dots,4$, are defined by
	\begin{equation}\label{weight function-2}
		w_1=U^{-(\alpha+1)},\ w_2=U^{-\alpha},\ w_3=U^{-2},\ w_4=U^{-1},
	\end{equation}
	with
	\begin{equation*}
		\alpha=\begin{cases}
			2p-1,\quad&{\rm if } \ \frac{1}{2}<p<1,\\
			0,\quad &{\rm if } \ 0<p\leq \frac{1}{2}.
		\end{cases}
	\end{equation*}
	Furthermore, the solution has the following asymptotic convergence
	\begin{equation}\label{2.12}
		\sup _{x \in \mathbb{R}}|(u, v)(x, t)-(U, V)(x-s t+x_0)| \rightarrow 0, \quad \text { as } t \rightarrow+\infty.
	\end{equation}
\end{theorem}

By using the Hopf-Cole transformation \eqref{hopf-cole}, we transfer the results of Theorems \ref{existence of TW} and \ref{transformed stability} back to the original chemotaxis model \eqref{orginal model}.

\begin{theorem}[Existence of traveling waves]\label{original TW existence}
	Assume that $u_{\pm}$ and $w_{\pm}$ satisfy \eqref{properties}. Then the system \eqref{orginal model} has a unique (up to a shift) smooth monotone traveling wave solution $(U,W)(z)$ satisfying
	\begin{equation}\label{Wz}
		U_z(z)<0,\quad W_z(z)>0,\quad \forall \ z\in\mathbb{R}.	
	\end{equation}
\end{theorem}

\begin{theorem}[Nonlinear stability]\label{original TW stability}
	Let $(U,W)(x-st)$ be the traveling wave solution obtained in Theorem \ref{original TW existence}. Then there exists a positive constant $\epsilon_{1}$ such that if $\left\|u_0-U\right\|_{1, w_3}+\left\|(\ln w_0)_x-(\ln W)_x\right\|_{1, w_4}+\left\|\phi_0\right\|_{w_1}+\left\|\psi_0\right\|_{w_2} \leq \epsilon_1$, and $\phi_0(+\infty)=\psi_0(+\infty)=0$, where
	$$
	\phi_0(x)=\int_{-\infty}^x(u_0(y)-U(y+x_0)) dy,\ \psi_0(x)=-\ln w_0(x)+\ln W(x+x_0),
	$$
	then the Cauchy problem \eqref{orginal model} with \eqref{initial data} has a unique global solution $(u,w)(x,t)$ satisfying
	\begin{equation}\nonumber
		u-U \in C\left([0, \infty) ; H_{w_3}^1\right),\
		w_x/w-W_x/W \in C\left([0, \infty) ; H_{w_4}^1\right), \	w-W\in C([0,\infty);H^2),
	\end{equation}
	and the following asymptotic convergence
	\begin{equation}
		\sup _{x \in \mathbb{R}}|(u, w)(x, t)-(U, W)(x-s t+x_0)| \rightarrow 0, \quad \text { as } t \rightarrow+\infty.
	\end{equation}
\end{theorem}

\section{Existence of traveling waves}\label{existence}

In this section, we first prove the existence of monotone traveling wave solution $(U,V)$ to the system \eqref{travleing wave equ}-\eqref{TW boundary data}. After that, we define the wave component $W$ through the second equation of \eqref{oritravleing wave equ} which yields the traveling wave solutions to the original chemotaxis system \eqref{orginal model}.

\begin{proof}[Proof of Theorem \ref{existence of TW}]
	Integrating \eqref{travleing wave equ} over $(z,+\infty)$, by \eqref{TW boundary data}, we have
	\begin{equation}\label{UV}
		\begin{cases}
			pU^{p-1}U_z=-s U-\chi U V, \\
			-s V=U,
		\end{cases}	
	\end{equation}
	which can be simplified to an ordinary differential equation (ODE) as
	\begin{equation}\label{3.2}
		U_z=\frac{U^{2-p}}{sp}\left(\chi U-s^2 \right).\end{equation}
	When $(u_-,v_-)$ satisfies \eqref{R-H condition}, by \eqref{s^2} one can see that
	\begin{equation}\label{f(U)}
		\frac{U^{2-p}}{sp}\left(\chi U-s^2 \right)=\frac{U^{2-p}}{sp}\chi\left( U-u_- \right).	
	\end{equation} Set $h(U)\triangleq\frac{U^{2-p}}{sp}\chi\left( U-u_- \right)$. Noting $h(U)<0$ for all $U\in(0,u_-)$, we only need to find the global solution of the following ODE
	\begin{equation}\label{ODE for U}
		\begin{cases}
			\frac{\mathrm{d}U}{\mathrm{d}z}=h(U), \\
			U(0)=\frac{u_-}{2},
		\end{cases}
	\end{equation}which is equivalent to
	\begin{equation}\label{eq1}
		\begin{cases}
			\frac{\mathrm{d}z}{\mathrm{d}U}=\frac{1}{h(U)} , \\
			z(\frac{u_-}{2})=0.
		\end{cases}	
	\end{equation}
	Integrating \eqref{eq1} gives
	\begin{equation*}
		z=\int_{\frac{u_-}{2}}^{U}\frac{1}{h(\tau)}d\tau\triangleq H(U).	
	\end{equation*}
	Clearly, for any given $U\in(0,u_-)$, the function $z=H(U)$ is monotonically decreasing. Moreover, since $0<p<1$, in view of \eqref{f(U)}, we get
	\[\lim_{U\rightarrow0^+}\int_{\frac{u_-}{2}}^{U}\frac{1}{h(\tau)}\mathrm{d}\tau=+\infty \ \text{ and } \ \lim_{U\rightarrow u_-^-}\int_{\frac{u_-}{2}}^{U}\frac{1}{h(\tau)}\mathrm{d}\tau=-\infty.\]
	Therefore, there exists a unique continuous function $H^{-1}$ such that
	\begin{equation*}
		U=H^{-1}(z),\quad \forall z\in(-\infty,+\infty).	
	\end{equation*}
	Using this $U$, one can define $V(z)$ through the second equation of \eqref{UV}. The monotonicity of $(U,V)$ follows from \eqref{f(U)}-\eqref{ODE for U} and the second equation of \eqref{UV}.
	
	We next investigate the convergence rates of $U(z)$ as $z\rightarrow\pm\infty$. Since $U\rightarrow 0$ as $z\rightarrow+\infty$, according to the asymptotic theory of ODE, the convergence rate of solutions to \eqref{ODE for U} as $z\rightarrow+\infty$ is determined by the equation
	\begin{equation*}
		\begin{cases}
			\frac{d\widehat{U}}{dz}=-\frac{\chi}{sp}u_-\widehat{U}^{2-p} , \\
			\widehat{U}(0)=\frac{u_-}{2}.
		\end{cases}	
	\end{equation*}
	By a direct calculation, we get
	$$
	\widehat{U}(z)=\left((\frac{u_-}{2})^{p-1}
	+\frac{\chi(1-p)}{sp}u_-z\right)^{-\frac{1}{1-p}} \text{ for }z>0,
	$$
	which indicates that
	\[U(z)\sim Cz^{-\frac{1}{1-p}}, \ \text{ as } z\rightarrow +\infty.\]
	On the other hand, the convergence rate of $U(z)$ as $z\rightarrow-\infty$ is determined by
	\begin{equation*}
		\begin{cases}
			\frac{d\overline{U}}{dz}=\frac{\chi}{sp}u_{-}^{2-p}\left( \overline{U}-u_- \right), \\
			\overline{U}(0)=\frac{u_-}{2}.
		\end{cases}	
	\end{equation*}
	It is easy to see that
	$$\overline{U}(z)=u_-- \frac{u_-}{2}\mathrm{e}^{\frac{\chi}{sp}u_{-}^{2-p}z} \text{ for }z<0.
	$$ Hence
	\begin{equation*}\begin{split}
			u_--U(z)&\sim C\mathrm{e}^{\frac{\chi}{sp}u_{-}^{2-p}z}, \ \text{ as } z\rightarrow -\infty.\end{split}\end{equation*}
	The proof is complete.
\end{proof}

\begin{remark}
	Since $U\in(0,u_-)$, it is easy to see from \eqref{ODE for U} that
	\begin{equation}\label{|Uz|estimate}
		|U_z(z)| \leq CU^{2-p}(z) \ \text{ for all }z\in(-\infty,+\infty).
	\end{equation}
	
\end{remark}

\begin{proof}[Proof of Theorem \ref{original TW existence}] The existence of $U(z)$ has been shown in Theorem \ref{existence of TW}. It remains to show the existence of bounded $W(z)$. By virtue of the second equation of \eqref{oritravleing wave equ}, it holds that
	$$
	W_z=\frac{1}{s}U(z)W(z)>0.
	$$
	According to Theorem \ref{existence of TW}, the integral $\int_0^{+\infty}U(y)dy$ is convergent, while  $\int_{-\infty}^{0}U(y)dy$ is divergent. Thus, we can solve $W$ as
	\begin{equation*}
		W(z)=w_{+} \exp \left(-\frac{1}{s} \int_z^{\infty} U(y) d y\right).	
	\end{equation*}
	This completes the proof of Theorem \ref{original TW existence}.
\end{proof}

\section{Nonlinear stability}\label{nonlinear stability}

In this section, we first study the nonlinear stability of the traveling wave solution $(U,V)$ for the parabolic-hyperbolic system \eqref{transformed model}-\eqref{transinitial data}. And then by using the Hopf-Cole transformation \eqref{hopf-cole}, we prove the nonlinear stability of the traveling wave solution $(U,W)$ for the original chemotaxis system \eqref{orginal model}.

\subsection{Reformulation of the problem}

Let $(U,V)(x-st)$ be the traveling wave obtained in Theorem \ref{existence of TW}. We write the system \eqref{transformed model} without viscosity in the form
\begin{equation*}
	\frac{\partial}{\partial t}\left[\begin{array}{l}
		u \\
		v
	\end{array}\right]+A\frac{\partial}{\partial x}\left[\begin{array}{l}
		u \\
		v
	\end{array}\right]=0,
\end{equation*}
where
\begin{equation*}
	A\triangleq\left[\begin{array}{cc}
		-\chi v & -\chi u \\
		-1 & 0
	\end{array}\right] .
\end{equation*}
Denote by $r_1\left(u_{-}, v_{-}\right)$ the first eigenvector of Jacobian matrix $A$. Then according to the theory of hyperbolic conservation laws, \cite{Smoller} the vectors $(u_-,v_-)^\top$ and $r_1\left(u_{-}, v_{-}\right)$ are linearly independent. Thus there are two numbers $x_0$ and $\gamma$ such that
\begin{equation*}
	\int_{-\infty}^{+\infty}\left(\begin{array}{l}
		u_0(x)-U(x) \\
		v_0(x)-V(x)
	\end{array}\right) d x=-x_0\left(\begin{array}{l}
		u_{-} \\
		v_{-}
	\end{array}\right)+\gamma r_1\left(u_{-}, v_{-}\right),
\end{equation*} provided that $u_0(x)-U(x)$ and $v_0(x)-V(x)$ are integrable over $(-\infty,+\infty)$. Thanks to the work of Liu and Zeng on the stability of viscous shock waves, \cite{Liu-Zeng09,Liu-Zeng15} the coefficient $\gamma$ may generate diffusion waves for small-amplitude viscous shock waves. In this paper, we aim to show the stability of \emph{large-amplitude} traveling waves and, for simplicity, consider the case $\gamma=0$. We leave the case $\gamma\neq0$ for future study. Now by the conservation laws, we have
\begin{equation}\label{shift}
	\begin{aligned}
		&\quad\int_{-\infty}^{+\infty}\left(\begin{array}{l}
			u(x, t)-U\left(x+x_0-s t\right) \\
			v(x, t)-V\left(x+x_0-s t\right)
		\end{array}\right) d x \\&=\int_{-\infty}^{+\infty}\left(\begin{array}{l}
			u_0(x)-U\left(x+x_0\right) \\
			v_0(x)-V\left(x+x_0\right)
		\end{array}\right) d x \\
		& =\int_{-\infty}^{+\infty}\left(\begin{array}{l}
			u_0(x)-U(x) \\
			v_0(x)-V(x)
		\end{array}\right) d x+\int_{-\infty}^{+\infty}\left(\begin{array}{l}
			U(x)-U\left(x+x_0\right) \\
			V(x)-V\left(x+x_0\right)
		\end{array}\right) d x \\
		& =\int_{-\infty}^{+\infty}\left(\begin{array}{l}
			u_0(x)-U(x) \\
			v_0(x)-V(x)
		\end{array}\right) d x+x_0\left(\begin{array}{l}
			u_{-} \\
			v_{-}
		\end{array}\right)\\&=\left(\begin{array}{l}
			0 \\
			0
		\end{array}\right).
	\end{aligned}
\end{equation}
We thus employ the anti-derivative technique to study the asymptotic stability of traveling wave $(U,V)$. Decompose the solution of \eqref{transformed model} as
\begin{equation}\label{decompose}
	(u, v)(x, t)=(U, V)\left(x+x_0-s t\right)+\left(\phi_z, \psi_z\right)(z, t),	
\end{equation}
where $z=x-st$. That is
\begin{equation}\label{u-U,v-V}
	(\phi, \psi)(z, t)=\int_{-\infty}^z\left(u(y, t)-U\left(y+x_0-s t\right), v(y, t)-V\left(y+x_0-s t\right)\right) d y,	
\end{equation}
for all $z\in \mathbb{R}$ and $t\geq0$. It then follows from \eqref{shift} that
\begin{equation}\label{phipsi initial data}
	\phi(\pm\infty,t)=\psi(\pm\infty,t)=0,\quad \text{for all }t>0.
\end{equation}
Without loss of generality, we assume that the translation $x_0=0$, which implies the initial perturbation is of zero mass
\begin{equation*}
	\int_{-\infty}^{+\infty}\left(\begin{array}{l}
		u_0(x)-U(x) \\
		v_0(x)-V(x)
	\end{array}\right) d x=\left(\begin{array}{l}
		0 \\
		0
	\end{array}\right).
\end{equation*}
Now substituting \eqref{decompose} into \eqref{transformed model}, integrating the system over $(-\infty,z)$, and using \eqref{travleing wave equ} and \eqref{phipsi initial data}, we derive the equation of $(\phi,\psi)$:
\begin{equation}\label{phipsi}
	\left\{\begin{array}{l}
		\phi_t=p\left(U^{p-1}\phi_{z}\right)_{z} +(s+\chi V) \phi_z+\chi U \psi_z+\chi \phi_z \psi_z+F_z,
		\\
		\psi_t=s \psi_z+\phi_z,
	\end{array}\right.
\end{equation}
where $F\triangleq \left(U+\phi_{z}\right)^{p}-U^p-pU^{p-1}\phi_{z}$, and the initial condition is given by
\begin{equation}\label{phi0psi0}
	\left(\phi_0, \psi_0\right)(z)\triangleq	\left(\phi, \psi\right)(z,0)=\int_{-\infty}^z\left(u_0-U, v_0-V\right)(y) d y,
\end{equation}
with $\left(\phi_0, \psi_0\right)( \pm \infty)=(0,0)$. One may observe that the  perturbation system  \eqref{phipsi} has a singular diffusivity near $z=+\infty$ since $U(+\infty)=0$ and $0<p<1$, in the sequel we shall carefully select weighted functional space to resolve the singularity.

\subsection{Energy estimates}

We search for the solutions of \eqref{phipsi}-\eqref{phi0psi0} in the following space
\begin{equation*}
	\begin{aligned}
		X(0,T):=\left\{ \left( \phi(z,t),\psi(z,t)\right)|\right.&\left.\phi\in C\left([0,T];L^{2}_{w_1}\right),\phi_z\in C\left([0,T];H^{1}_{w_3}\right),\right.\\&
		\left.\psi\in C\left([0,T];L^{2}_{w_2}\right),\psi_z\in C\left([0,T];H^{1}_{w_4}\right)\right\},		
	\end{aligned}
\end{equation*}
where the weight functions $w_i$, $i=1,\dots,4$, are defined by
\begin{equation}\label{weight function}
	w_1=U^{-(\alpha+1)},\ w_2=U^{-\alpha},\ w_3=U^{-2},\ w_4=U^{-1},
\end{equation}
with
\begin{equation*}
	\alpha=\begin{cases}
		2p-1,\quad&{\rm if } \ \frac{1}{2}<p<1,\\
		0,\quad &{\rm if } \ 0<p\leq \frac{1}{2}.
	\end{cases}
\end{equation*}
Denote
\begin{equation}\label{N(t)}
	N(t)\triangleq\sup_{\tau\in [0,t]}\left(\left\|\phi(\cdot,\tau)\right\|_{w_1}+\left\|\psi(\cdot,\tau)\right\|_{w_2}+\left\|\phi_z(\cdot,\tau)\right\|_{1,w_3} +\left\|\psi_z(\cdot,\tau)\right\|_{1,w_4}\right).	
\end{equation}
Then we have the following global well-posedness result for the Cauchy problem \eqref{phipsi}-\eqref{phi0psi0}.

\begin{proposition}\label{phipsi stability}
	Suppose $(\phi_0,\psi_0)\in L_{w_1}^2(\mathbb{R})\times L_{w_2}^2(\mathbb{R})$ and that $(\phi_{0z},\psi_{0z})\in H_{w_3}^{1}(\mathbb{R})\times H_{w_4}^{1}(\mathbb{R}) $. Then there exists a positive constant $\delta_0$ such that if $N(0)\leq \delta_0$, then the Cauchy problem \eqref{phipsi}-\eqref{phi0psi0} has a unique global solution $(\phi,\psi)\in X(0,\infty)$ satisfying
		\begin{align}
			\|\phi\|_{w_1}^2\!&+\!\|\psi\|_{w_2}^2\!+\!\left\|\phi_z\right\|_{1, w_3}^2\!+\!\left\|\psi_z\right\|_{1, w_4}^2\!+\!\int_0^t\left(\left\|\phi_z(\tau)\right\|_{w_5}^2\!+\!\left\|\phi_{zz}(\tau)\right\|_{1, w_6}^2\!+\!\left\|\psi_z(\tau)\right\|_{1}^2\right) d \tau\nonumber \\
			&\quad\leq C\left(\|\phi_0\|_{w_1}^2+\|\psi_0\|_{w_2}^2+\left\|\phi_{0z}\right\|_{1, w_3}^2+\left\|\psi_{0z}\right\|_{1, w_4}^2\right) \leq C N^2(0),\label{priori estimate}
		\end{align}
	for any $t\in [0,\infty)$, where $w_i$, $i=1,\dots,4$, are defined by \eqref{weight function}, $w_5=U^{p-\alpha-2}$ and $w_6=U^{p-3}$.
	
\end{proposition}

The local well-posedness of the system \eqref{phipsi}-\eqref{phi0psi0} is standard (see Ref. \cite{local-existence} for instance). To prove Proposition \ref{phipsi stability}, we only need to establish the following \textit{a priori} estimate.

\begin{proposition}[\emph{A priori} estimate] \label{proposition priori estimate}
	Assume that $(\phi,\psi)\in X(0,T)$ is a solution of \eqref{phipsi}-\eqref{phi0psi0} for a positive constant $T$. Then there exists a positive constant $\delta_1$, independent of $T$, such that if
	\begin{equation*}
		N(t)\leq \delta_1 \ \text{ for all }	t\in[0,T],
	\end{equation*}
	then the estimate \eqref{priori estimate} holds for all $t\in [0,T]$.
\end{proposition}
Before establishing the \textit{a priori} estimate in Proposition \ref{proposition priori estimate}, we first present some preliminary calculations.

\begin{lemma}
	Let $U(z)$ be the traveling wave obtained in Theorem \ref{existence of TW}. Denote $w(U)\triangleq U^{-\beta}$ for some positive constant $\beta$, then it holds
	\begin{equation}\label{weighted sobolev inequ}
		\sup_{z \in \mathbb{R}}|\sqrt{w(U)}f|^2\leq C\|f\|^2_{1,w}, \ \forall f\in H_w^1(\mathbb{R}),
	\end{equation}
	where $C$ is a positive constant  independent of $f$.
\end{lemma}
\begin{proof}
	By \eqref{|Uz|estimate}, noting $0<p<1$ and $U\in(0,u_-)$, it holds that for any $f\in H_w^1(\mathbb{R})$
	\begin{equation}\nonumber
		\begin{aligned}
			\int |(\sqrt{w(U)})_z|^2f ^2=\frac{\beta^2}{4}\int \frac{U_z^2}{U^{2}} \cdot w(U)f^2 \leq C\int w(U)f^2.
		\end{aligned}
	\end{equation}
	Thus,
		\begin{align}\|\sqrt{w(U)}f\|_1^2&\leq\int w(U)f^2+2\int w(U)f_z^2+2\int |(\sqrt{w(U)})_z|^2f^2\nonumber\\&\leq C\left(\int w(U)f^2+\int w(U)f_z^2\right)\nonumber\\&\leq C\|f\|^2_{1,w}.\label{eq32}
	\end{align}
	According to the Sobolev inequality, we have
	$$
	\sup_{z \in \mathbb{R}}|\sqrt{w(U)}f|^2\leq C\|\sqrt{w(U)}f\|_1^2.
	$$
	This along with \eqref{eq32} gives rise to \eqref{weighted sobolev inequ}.
\end{proof}

By the weighted Sobolev inequality \eqref{weighted sobolev inequ}, it is easy to see that
\begin{equation}\label{sobolev embedding}
	\sup_{\tau\in [0,t]}\left\{\left\|\sqrt{w_1}\phi(\cdot,\tau)\right\|_{L^{\infty}},
	\left\|\sqrt{w_3}\phi_z(\cdot,\tau)\right\|_{L^{\infty}}, \left\|\sqrt{w_4}\psi_z(\cdot,\tau)\right\|_{L^{\infty}}\right\} \leq CN(t).
\end{equation}

\begin{lemma}\label{taylor}
	Let the assumptions of Proposition \ref{proposition priori estimate} hold. If $N(T)\ll1$, then it holds that
	\begin{equation}\label{F}
		|F|\leq C U^{p-2}\phi_{z}^2,
	\end{equation}	
	\begin{equation}\label{Fz}
		|F_z|\leq C\left( U^{p-2} |\phi_{z}||\phi_{zz}|+U^{p-3}|U_z|\phi_{z}^2\right),
	\end{equation}
	\begin{equation}\label{Fzz}
		|F_{zz}| \leq C\left[U^{p-2} \phi_{z z}^2+\left(U^{p-4}U_z^2+\frac{|U_z|}{U}\right)\phi_z^2 +U^{p-3} |U_z| |\phi_z| |\phi_{z z}|+U^{p-2} |\phi_z| |\phi_{z z z}|\right].
	\end{equation}

\end{lemma}
\begin{proof}
	A direct calculation yields
	\begin{equation}\nonumber
		\begin{aligned}
			F_z&=p\left(U+\phi_{z}\right)^{p-1}\left( U_z+\phi_{zz}\right)-pU^{p-1}U_z-p(p-1)U^{p-2}U_z\phi_{z}-pU^{p-1}\phi_{z z}\\
			&=p\left[\left(U+\phi_{z}\right)^{p-1}-U^{p-1} \right]\phi_{z z}+p\left[\left(U+\phi_{z}\right)^{p-1}-U^{p-1}-(p-1)U^{p-2}\phi_{z} \right]U_z.\\	
		\end{aligned}
	\end{equation}	
	Owing to \eqref{sobolev embedding}, we get $\left\|\phi_z/U\right\|_{L^{\infty}}\leq CN(t)$. Therefore, \eqref{F} and \eqref{Fz} can be immediately obtained from the Taylor's expansion. Similarly, by \eqref{3.2}, one can derive
	\begin{equation}\nonumber
		\begin{aligned}
			F_{zz}&= p(p\!-\!1)\left(U\!+\!\phi_z\right)^{p\!-\!2} \left( U_z\!+\!\phi_{z z}\right)^2\!+\!p\left(U\!+\!\phi_z\right)^{p-1}\left(U_{z z}\!+\!\phi_{z z z}\right)\!-\!p(p-1) U^{p-2} U_z^2\\&\quad -p U^{p-1} U_{z z}-p(p-1)(p-2) U^{p-3} U_z^2 \phi_z-p(p-1) U^{p-2} U_{z z} \phi_z\\&\quad-p(p-1) U^{p-2} U_z \phi_{z z}  -p(p-1) U^{p-2} U_z \phi_{z z}-p U^{p-1} \phi_{z z z} \\
			&= p(p\!-\!1)\left(U\!+\!\phi_z\right)^{p\!-\!2} \phi_{z z}^2\!+\!p(p\!-\!1)\left[\left(U\!+\!\phi_z\right)^{p\!-2}\!-\!U^{p\!-\!2}\!-\!(p\!-\!2) U^{p\!-\!3} \phi_z\right] U_z^2 \\
			&\quad +\!p\left[\left(U\!+\!\phi_z\right)^{p-1}\!-\!U^{p-1}\!-\!(p\!-\!1) U^{p-2} \phi_z\right] U_{z z}\!+\!p\left[\left(U+\phi_z\right)^{p-1}-U^{p-1}\right] \phi_{z z z} \\
			&\quad + 2 p(p-1)\left[\left(U+\phi_z\right)^{p-2}-U^{p-2}\right] U_z \phi_{z z}\\
			&=p(p\!-\!1)\left(U\!+\!\phi_z\right)^{p\!-\!2} \phi_{z z}^2\!+\!p(p\!-\!1)\left[\left(U\!+\!\phi_z\right)^{p\!-\!2}\!-\!U^{p-2}\!-\!(\!p-\!2) U^{p-3} \phi_z\right] U_z^2 \\
			&\quad +p\left[\left(U+\phi_z\right)^{p-1}-U^{p-1}-(p-1) U^{p-2} \phi_z\right]\left[\frac{\chi}{sp}U^{2-p}U_z+(2-p)\frac{{U_z}^2}{U} \right]\\
			&\quad +p\left[\left(U+\phi_z\right)^{p-1}-U^{p-1}\right] \phi_{z z z} + 2 p(p-1)\left[\left(U+\phi_z\right)^{p-2}-U^{p-2}\right] U_z \phi_{z z},
		\end{aligned}
	\end{equation}	
	which implies
	\begin{equation}\nonumber
		\begin{aligned}
			\!|F_{zz}\!|\!\leq\! C \left| U^{\!p-\!2} \phi_{z z}^2\!+\!U^{p\!-\!4}U_z^2\phi_z^2\! +\!U^{p\!-\!3}(U^{2\!-\!p}U_z\!+\!\frac{U_z^2}{U}) \phi_z^2\!+\!U^{p\!-\!2} \phi_z \phi_{z z z}\!+\!U^{p\!-\!3} U_z \phi_z \phi_{z z}\right|.
		\end{aligned}
	\end{equation}
	Hence \eqref{Fzz} holds and we complete the proof of Lemma \ref{taylor}.
\end{proof}

We now derive the basic $L^2$ estimate of $(\phi,\psi)$.
\begin{lemma}\label{L2}
	Let the assumptions of Proposition \ref{proposition priori estimate} hold. If $N(T)\ll1$, then there exists a constant $C>0$ independent of $T$ such that
	\begin{equation}\label{L2 estimate}
		\|\phi\|_{w_1}^2+\|\psi\|_{w_2}^2+\int_0^t\left\|\phi_z(\tau)\right\|_{w_5}^2  \leq C\left(\|\phi_0\|_{w_1}^2+\|\psi_0\|_{w_2}^2\right) +C N(t)\int_0^t \int U^{1-p}\psi_z^2,
	\end{equation}
	for any $t\in [0,T]$.
\end{lemma}

\begin{proof}
	Denote by $w(U)$ a smooth positive weight function. Multiplying the first equation of \eqref{phipsi} by $\frac{w(U)}{U}\phi$ and the second one by $\chi w(U)\psi$, adding them and integrating by parts, noting that
	$$
	\frac{w(U)}{U}\left( U^{p-1}\phi_z\right)_z\phi=\left( U^{p-1}\frac{w(U)}{U}\phi\phi_{z}\right) _z-U^{p-1}\left(\frac{w(U)}{U} \right)_z\phi\phi_{z}-U^{p-1}\frac{w(U)}{U}\phi^2_{z},
	$$
	and
	$$
	\frac{w(U)\left(s\!+\!\chi V\right)}{U}\phi \phi_z\!=\!\left(\frac{ w(U)\left(s\!+\!\chi V\right)}{2U}\phi^2\right)_z\!-\!\frac{\left(s\!+\!\chi V\right)_z}{2}\frac{w(U)}{U}\phi^2\!-\!\frac{\left(s\!+\!\chi V\right)}{2}\left(\frac{w(U)}{U}\right)_z\phi^2,
	$$	
	we obtain
	\begin{equation}\label{eq2}
		\begin{aligned}
			&\frac{1}{2}\frac{\mathrm{d}}{\mathrm{d}t}\int\left(\frac{w(U)}{U}\phi^2+\chi w(U)\psi^{2} \right)+p\int U^{p-2} w(U) \phi^2_{z}\\&+\frac{1}{2}\int\left[\left(s+\chi V\right)_z\frac{w(U)}{U}+\left(s+\chi V\right)\left(\frac{w(U)}{U}\right)_z \right] \phi^2+\frac{s\chi}{2}\int w(U)_z\psi^{2}\\&+\chi\int w(U)_z\phi\psi=\chi\int\frac{w(U)}{U}\phi\phi_{z}\psi_z
			+\int\frac{w(U)}{U}F_z\phi-p\int U^{p-1}\left(\frac{w(U)}{U}\right)_z\phi\phi_{z}.
		\end{aligned}	
	\end{equation}
	By the Cauchy-Schwarz inequality, the last term on the right hand side of \eqref{eq2} satisfies
	$$
	p\left|\int U^{p-1}\left(\frac{w(U)}{U}\right)_z\phi\phi_{z}\right|\leq \eta p\int U^{p-2} w(U) \phi_{z}^2+\frac{p}{4\eta}\int\left[\left(\frac{w(U)}{U}\right)_z \right]^2\frac{ U^{p}}{w(U)}\phi^2,
	$$
	where $\eta$ is a positive constant to be determined later. It thus follows from \eqref{eq2} that
		\begin{align}
			&\frac{1}{2}\frac{\mathrm{d}}{\mathrm{d}t}\int\left(\frac{w(U)}{U}\phi^2+\chi w(U)\psi^{2} \right)+(1-\eta)p\int U^{p-2} w(U) \phi^2_{z}\nonumber\\&+\frac{1}{2}\int z_\eta(U)\frac{w(U)}{U}\phi^2+\frac{s\chi}{2}\int w(U)_z\psi^{2}+\chi\int w(U)_z\phi\psi\nonumber\\&\leq\chi\int\frac{w(U)}{U}\phi\phi_{z}\psi_z
			+\int\frac{w(U)}{U}F_z\phi,\label{eq3}
		\end{align}	
	where
	\begin{equation*}
		z_\eta(U)\triangleq\left(s+\chi V\right)_z+\left(s+\chi V\right)\left(\frac{w(U)}{U}\right)_z \frac{U}{w(U)}-\frac{p}{2\eta}\left[\left(\frac{w(U)}{U}\right)_z \right]^2 \frac{U^{p+1}}{w^2(U)}.
	\end{equation*}
	
	We next determine the weight function $w(U)>0$ and the constant $\eta\in(0,1)$ so that $z_\eta(U)\geq0$ and $w(U)_z\geq0$ for all $U\in(0,u_-)$. To ensure $w(U)_z\geq0$, given that $U_z<0$ according to \eqref{Uz}, we may take
	\begin{equation}\label{w(U)}
		w(U)=U^{-\alpha}, \quad \forall \ U\in (0,u_-),
	\end{equation}
	where $\alpha\geq0$ is a constant to be determined later. Given this, by \eqref{UV}-\eqref{3.2}, one can see that
		\begin{align}
			z_\eta(U)&=-\frac{\chi}{s}U_z+\left(s-\frac{\chi}{s}U\right)\left( U^{-\alpha-1}\right)_zU^{1+\alpha}-\frac{p}{2\eta}\left[\left( U^{-\alpha-1}\right)_z \right]^2U^{p+1+2\alpha}\nonumber\\
			&=-\frac{\chi}{s}U_z-\frac{\alpha+1}{s}\frac{U_z}{U}(s^2-\chi U)-\frac{p(\alpha+1)^2}{2\eta}U^{p-3}U_z^2\nonumber\\&=-\frac{\chi}{s}U_z+p(\alpha+1)U^{p-3}U_z^2-\frac{p(\alpha+1)^2}{2\eta}U^{p-3}U_z^2\nonumber\\&=-\frac{\chi}{s}U_z+(1-\frac{\alpha+1}{2\eta})p(\alpha+1)U^{p-3}U_z^2.\label{eq45}
		\end{align}
Now we take
	\begin{equation}\label{alpha}
		\alpha=\begin{cases}
			2p-1,\quad&{\rm if } \ \frac{1}{2}<p<1,\\
			0,\quad &{\rm if } \ 0<p\leq \frac{1}{2},
		\end{cases}	\text{ and } \eta=\begin{cases}
			p,\quad&{\rm if } \ \frac{1}{2}<p<1,\\
			\frac{1}{2},\quad &{\rm if } \ 0<p\leq \frac{1}{2},
		\end{cases}
	\end{equation}
	so that $1-\frac{\alpha+1}{2\eta}=0$. Therefore, by virtue of $s>0$, \eqref{Uz} and \eqref{eq45}, we get
	\begin{equation}\label{eq5}
		z_\eta(U)=-\frac{\chi}{s}U_z>0,\quad\forall \ U \in (0,u_-).
	\end{equation}

	On the other hand, utilizing \eqref{F}, we can estimate the last term on the right-hand side of \eqref{eq3} after integration by parts as
		\begin{align}
			\int\frac{w(U)}{U}F_z\phi&=-\int F\left[U^{-\alpha-1}\phi_{z}+\left(U^{-\alpha-1}\right)_{z}\phi \right]\nonumber\\
			&\leq C\left[\int U^{p-\alpha-3}|\phi_z|^3+\int U^{p-\alpha-4}\left| U_z\right| |\phi|\phi_z^2\right].\label{eq6}
		\end{align}
	Thanks to \eqref{|Uz|estimate} and \eqref{w(U)}, along  with the fact that $\|\phi/U^{\frac{\alpha+1}{2}}\|_{L^{\infty}}\leq CN(t)$, it holds that
		\begin{align}
			\int U^{p-\alpha-4}\left| U_z\right| |\phi|\phi_z^2&\leq CN(t)\int U^{p-\alpha-4}U^{\frac{\alpha+1}{2}}|U_z|	\phi_{z}^2\nonumber\\&\leq CN(t)p\int U^{\frac{\alpha+1-2p}{2}} U^{p-\alpha-2}\phi_{z}^2\nonumber\\&\leq	CN(t)p\int  U^{p-\alpha-2}\phi_{z}^2,\label{eq4}\end{align}
	due to $\alpha+1-2p\geq0$.
	Now inserting \eqref{eq5}-\eqref{eq4} into \eqref{eq3} yields
		\begin{align}
			&\frac{1}{2}\frac{\mathrm{d}}{\mathrm{d}t}\int\left(\frac{\phi^2 }{U^{\alpha+1}}+\chi\frac{\psi^{2}}{U^{\alpha}}  \right)+(1-\eta-CN(t))p\int U^{p-\alpha-2}\phi^2_{z}\nonumber\\&-\frac{\chi}{2s}\int U^{-\alpha-1}U_z\phi^2-\chi\alpha\int U^{-\alpha-1}U_z\phi\psi-\frac{s\chi\alpha}{2}\int U^{-\alpha-1}U_z\psi^{2}\nonumber\\&\leq\chi\int U^{-\alpha-1}\phi\phi_{z}\psi_z+C\int U^{p-\alpha-3} |\phi_{z}|^{3},\label{eq7}
		\end{align}	
	where
		\begin{align}
			&\quad-\frac{\chi}{2s}\int U^{-\alpha-1}U_z\phi^2-\chi\alpha\int U^{-\alpha-1}U_z\phi\psi-\frac{s\chi\alpha}{2}\int U^{-\alpha-1}U_z\psi^{2}\nonumber\\
			&=-\frac{\chi}{2s}\int \left(\phi^2+2s\alpha \phi\psi+s^2\alpha\psi^{2} \right)U^{-\alpha-1}U_z.\label{positive definite}
		\end{align}
	We claim that
		\begin{align}
			g(\phi,\psi)&\triangleq\phi^2+2s\alpha \phi+s^2\alpha\psi^{2}\nonumber\\
			&=\begin{bmatrix}
				\phi&\psi
			\end{bmatrix}\begin{bmatrix}
				1&s\alpha\\s\alpha&s^2\alpha
			\end{bmatrix}\begin{bmatrix}
				\phi\\\psi
			\end{bmatrix}\nonumber\\
			&\triangleq Y^\top B Y> 0.\label{eq8}
		\end{align}
	In fact, when $\frac{1}{2}<p<1$, we have $0<\alpha<1$ due to \eqref{alpha}. And then the two eigenvalues of $B$ satisfies $0<\lambda_1<\lambda_2$.
	Therefore, $B$ is positive definite and the claim is true for $\frac{1}{2}<p<1$. When $0<p\leq\frac{1}{2}$, thanks to \eqref{alpha}, we have from \eqref{eq8} that $g(\phi,\psi)=\phi^2>0$ for any $ \phi\neq 0$. Hence in any case, \eqref{eq8} is true, which along with \eqref{positive definite} gives rise to
	\begin{equation}\label{eq9}
		-\frac{\chi}{2s}\int U^{-\alpha-1}U_z\phi^2-\chi\alpha\int U^{-\alpha-1}U_z\phi\psi-\frac{s\chi\alpha}{2}\int U^{-\alpha-1}U_z\psi^{2}> 0.
	\end{equation}
	Substituting \eqref{eq9} into \eqref{eq7}, we then arrive at
		\begin{align}
			&\frac{1}{2}\frac{\mathrm{d}}{\mathrm{d}t}\int\left(\frac{\phi^2 }{U^{\alpha+1}}+\chi\frac{\psi^{2}}{U^{\alpha}}  \right)+(1-\eta-CN(t))p\int U^{p-\alpha-2}\phi^2_{z}\nonumber\\&\leq\chi\int U^{-\alpha-1}\phi\phi_{z}\psi_z+C\int U^{p-\alpha-3} |\phi_{z}|^{3}.\label{eq10}
		\end{align}	
	Noting $\|\phi/U^{\frac{\alpha+1}{2}}\|_{L^{\infty}}\leq CN(t)$, by the Cauchy-Schwarz inequality, we have
		\begin{align}
			\chi\int U^{-\alpha-1}\phi\phi_{z}\psi_z&\leq CN(t)\int U^{-\alpha-1} U^{\frac{\alpha+1}{2}}|\phi_{z}||\psi_z|\nonumber\\&\leq CN(t)p\int U^{p-\alpha-2}\phi_{z}^2 +CN(t)\int U^{1-p}\psi_{z}^2.\label{eq11}
		\end{align}
	And the last term on the right hand side of \eqref{eq10} can be estimated as
	\begin{equation}\label{eq12}
		\int U^{p-\alpha-3} |\phi_{z}|^{3}=\int U^{-1} |\phi_{z}| U^{p-\alpha-2} \phi_{z}^{2}\leq  CN(t)p\int U^{p-\alpha-2} \phi_{z}^{2},	
	\end{equation}
	where we have used $\|\phi_{z}/U\|_{L^\infty}\leq CN(t)$. Combining \eqref{eq11} and \eqref{eq12}, after choosing $N(t)$ suitably small, we thus update \eqref{eq10} as
	\begin{equation}\label{eq13}
		\frac{\mathrm{d}}{\mathrm{d}t}\int\left(\frac{\phi^2 }{U^{\alpha+1}}+\frac{\psi^{2}}{U^{\alpha}}  \right)+\int U^{p-\alpha-2}\phi^2_{z}\leq CN(t)\int U^{1-p}\psi_{z}^2.
	\end{equation}
	Then integrating \eqref{eq13} over $(0,t)$ gives the desired estimate \eqref{L2 estimate}.
\end{proof}

The next lemma gives the $H^1$ estimate of $(\phi,\psi)$. An issue in the weighted energy estimates for higher order derivatives is that the singularity caused by fast diffusion may generate non-zero boundary terms during the integration by parts. To break down this barrier, we develop an approximate procedure that avoids the boundary terms after integration by parts, and obtain the desired weighted estimates by employing the Fatou's Lemma.
\begin{lemma}\label{H1}
	Let the assumptions of Proposition \ref{proposition priori estimate} hold. If $N(T)\ll1$, then it holds
		\begin{equation}\begin{split}			&\|\phi_z\|_{w_3}^2+\|\psi_z\|_{w_4}^2+\int_0^t\left(\left\|\phi_{zz}(\tau)\right\|_{w_6}^2+\left\|\psi_{z}(\tau)\right\|^2 \right)\\& \leq C\left(\|\phi_0\|_{w_1}^2+\|\psi_0\|_{w_2}^2+\|\phi_{0z}\|_{w_3}^2
			+\|\psi_{0z}\|_{w_4}^2\right).\label{H1 estimate}
\end{split}		\end{equation}
	
\end{lemma}

\begin{proof}
	Differentiating \eqref{phipsi} with respect to $z$ yields
	\begin{equation}\label{phizpsiz}
		\begin{cases}
			\phi_{z t}=&\left[pU^{p-1}\phi_{zz}+p(p-1)U^{p-2}U_z\phi_{z}\right]_z+\left( s+\chi V\right)\phi_{z z}+\chi V_z \phi_z\\&+\chi U_z \psi_z+\chi U \psi_{z z}+\chi\left(\phi_z \psi_z\right)_z+F_{zz}, \\
			\psi_{z t}=&s \psi_{z z}+\phi_{z z}.
		\end{cases}
	\end{equation}
	Denote $U_\epsilon\triangleq U+\epsilon$, where $\epsilon>0$ is a constant. Multiplying the first equation of \eqref{phizpsiz} by $\frac{\phi_{z}}{UU_\epsilon}$ and the second one by $\chi\frac{\psi_{z}}{U_\epsilon}$, integrating the resultant equations with respect to $z$, noting
	\begin{equation}\nonumber
		\begin{aligned}
			&\quad\int \left[pU^{p-1}\phi_{zz}+p(p-1)U^{p-2}U_z\phi_{z}\right]_z\frac{\phi_z}{UU_\epsilon}\\
			&=-p\int \frac{U^{p-2} }{U_\epsilon}\phi_{z z}^2+p\left(2-p\right)\int \frac{U^{p-3}U_z}{U_\epsilon} \phi_{z}\phi_{zz}+p\int\frac{U^{p-2}{U}_z}{U_\epsilon^2}\phi_{z}\phi_{zz}\\
			&\quad+p(p-1)\int\left(\frac{U^{p-4}U_z^2 }{U_\epsilon} +\frac{U^{p-3}{U_z^2} }{U_\epsilon^2}\right) \phi_{z}^2,
		\end{aligned}
	\end{equation}
	and
	\begin{equation}\nonumber
		\chi\int\frac{\phi_{z}\psi_{z z}}{U_\epsilon}+\chi\int\frac{\psi_{z}\phi_{z z}}{U_\epsilon}=\chi\int\frac{\left(\phi_{z}\psi_{z} \right)_z}{U_\epsilon}=\chi\int\frac{{U}_z}{U_\epsilon^2}\phi_{z}\psi_{z},
	\end{equation}
	we get
		\begin{align}
			&\frac{1}{2}\frac{\mathrm{d}}{\mathrm{d}t}\int\left(\frac{\phi_z^2}{UU_\epsilon}+\chi \frac{\psi_z^2}{U_\epsilon} \right)+p\int \frac{U^{p-2} }{U_\epsilon}\phi_{zz}^2-\frac{s\chi}{2}\int\frac{{U}_z}{U_\epsilon^2}\psi_z^2
			\nonumber\\&+p(1-p)\int \left(\frac{U^{p-4}U_z^2 }{U_\epsilon}+\frac{U^{p-3}|U_z|^2 }{U_\epsilon^2}\right) \phi_{z}^2\nonumber\\
			&=\int \left[p(2-p)\frac{U^{p-3}U_z}{U_\epsilon}+p \frac{U^{p-2}{U}_z}{U_\epsilon^2}+\frac{s+\chi V}{UU_\epsilon}\right] \phi_{z}\phi_{zz}+\chi\int\frac{V_z}{UU_\epsilon}\phi_{z}^2\nonumber \\
			&+\chi\int\left(\frac{{U}_z}{U_\epsilon^2}
			+\frac{U_z}{UU_\epsilon}\right)\phi_{z}\psi_{z}+\chi\int\frac{\left(\phi_{z}\psi_{z} \right)_z\phi_z}{UU_\epsilon}+\int \frac{F_{zz}\phi_z}{UU_\epsilon}.\label{eq14}
		\end{align}
	Next, we estimate the terms on the right hand side of \eqref{eq14}. Noting $\frac{1}{U_\epsilon}\leq \frac{1}{U}$ and $|s+\chi V|\leq C$, by \eqref{|Uz|estimate} and the Cauchy-Schwarz inequality, we deduce that
		\begin{align}
			&\quad \int \left[p(2-p)\frac{U^{p-3}U_z}{U_\epsilon}+p \frac{U^{p-2}{U}_z}{U_\epsilon^2}+\frac{s+\chi V}{UU_\epsilon}\right] \phi_{z}\phi_{zz}\nonumber\\&\leq C\left(\int \frac{|\phi_{z}||\phi_{z z}|}{UU_\epsilon}+\int \frac{|\phi_{z}||\phi_{z z}|}{U_\epsilon^2} \right)\nonumber\\&\leq \frac{p}{4}\int \frac{U^{p-2} }{U_\epsilon}\phi_{zz}^2+C\left(\int\frac{\phi_{z}^2 }{U^pU_\epsilon}+\int \frac{U^{2-p} }{U_\epsilon^3}\phi_{z}^2 \right)\nonumber\\&\leq\frac{p}{4}\int \frac{U^{p-2} }{U_\epsilon}\phi_{zz}^2+C\int\frac{\phi_{z}^2 }{U^{p+1}}.\label{eq15}
		\end{align}	
	Similarly, by the second equation of \eqref{travleing wave equ}, we have
	\begin{equation}\label{4.40}\chi\left|\int\frac{V_z}{UU_\epsilon}\phi_{z}^2\right|= \frac{\chi}{s}\left|\int\frac{U_z}{UU_\epsilon}\phi_{z}^2\right|\leq C\int \frac{\phi_{z}^2}{U^p},
	\end{equation}
	and
		\begin{align}
			\chi\int\left(\frac{{U}_z}{U_\epsilon^2}			+\frac{U_z}{UU_\epsilon}\right)\phi_{z}\psi_{z}&\leq\frac{s\chi}{4}\int\frac{|{U}_z|}{U_\epsilon^2}\psi_{z}^2+C\left(\int\frac{|{U}_z|}{U_\epsilon^2}\phi_{z}^2
+\int\frac{|U_z|}{U^2}\phi_{z}^2 \right)\nonumber\\&\leq\frac{s\chi}{4}\int\frac{|{U}_z|}{U_\epsilon^2}\psi_{z}^2+C\int \frac{\phi_{z}^2}{U^p}.\label{4.41}
		\end{align}
	In view of \eqref{|Uz|estimate} and the fact that $\|\phi_{z}/U\|_{L^\infty}\leq CN(t)$, by integration by parts, we get
	\begin{equation}\label{4.42}
		\begin{aligned}
			\!\chi\!\int\frac{\left(\phi_{z}\psi_{z} \right)_z\phi_z}{UU_\epsilon}&=-\chi\int \frac{\phi_z\phi_{zz}\psi_{z}}{UU_\epsilon}+\chi\int\frac{U_z}{U^2U_\epsilon}\phi_z^2\psi_{z}+\chi\int\frac{{U} _z}{U_\epsilon^2U}\phi_z^2\psi_{z}\\&\leq\! CN(t)\!\int\left(\frac{|\phi_{zz}||\psi_{z}|}{U_\epsilon}
			\!+\frac{|U_z|}{UU_\epsilon}|\phi_z||\psi_{z}|
			\!+\frac{|{U}_z|}{U_\epsilon^2}|\phi_z||\psi_{z}|\right)\\&\!\leq\!CN(t)\int \left(\frac{U^{p\!-\!2}}{U_\epsilon}\phi_{zz}^2
			\!+\frac{|{U}_z|}{U_\epsilon^2}\psi_{z}^2\!+U^{1\!-\!p}\psi_z^2\right)\!+\!C\int\frac{\phi_{z}^2}{U^p},
		\end{aligned}
	\end{equation}
	and by \eqref{Fz}, the last term on the right hand side of \eqref{eq14} can be estimated as
		\begin{align}
			\int \frac{F_{zz} \phi_z}{UU_\epsilon}&=-	\int F_{z}\left(\frac{\phi_{zz}}{UU_\epsilon}-\frac{U_z\phi_{z}}{U^2U_\epsilon}-\frac{{U}_z\phi_{z}}{U_\epsilon^2U}\right)\nonumber\\&\leq C\left[\int \frac{U^{p-3} }{U_\epsilon} |\phi_{z}|\phi_{zz}^2+\int\left(\frac{U^{p-4}|U_z| }{U_\epsilon}+\frac{U^{p-3}|{U}_z| }{U_\epsilon^2}\right)  \phi_{z}^2|\phi_{zz}|\right.\nonumber\\&\quad\left.+\int\left(\frac{U^{p-5}U_z^2 }{U_\epsilon}+\frac{U^{p-4}|U_z|^2 }{U_\epsilon^2}\right)  |\phi_{z}|^3\right]\nonumber\\&\leq CN(t)p\int\frac{U^{p-2}}{U_\epsilon} \phi_{zz}^2
			+CN(t)\int\frac{{\phi_{z}^2}}{U^{p+1}}.\label{eq17}
		\end{align}
	Substituting \eqref{eq15}--\eqref{eq17} into \eqref{eq14} and integrating the equation with respect to $t$, noting $\frac{1}{U^{p}}\leq CU^{p-\alpha-2}$, $\frac{1}{U^{p+1}}\leq CU^{p-\alpha-2}$ and $U_z<0$, we have
		\begin{align}
			&\int\left(\frac{\phi_z^2}{UU_\epsilon}+ \frac{\psi_z^2}{U_\epsilon} \right)+\int_{0}^{t}\int\left( \frac{U^{p-2}}{U_\epsilon}\phi_{zz}^2+\frac{|{U}_z|}{U_\epsilon^2}\psi_z^2\right) \nonumber\\&\leq C\int\left(\frac{\phi_{0z}^2}{U^2}+ \frac{\psi_{0z}^2}{U}+\frac{\phi_{0}^2 }{U^{\alpha+1}}+\frac{\psi_{0}^2 }{U^{\alpha}} \right)+CN(t)\int_{0}^{t}\int U^{1-p}\psi_z^2\nonumber,	 \end{align}
	provided that $N(t)$ is small enough. Moreover, it follows from the Fatou's Lemma that
		\begin{align}
			&\quad\int\left(\frac{\phi_z^2}{U^2}+ \frac{\psi_z^2}{U} \right)+\int_{0}^{t}\int\left( U^{p-3} \phi_{zz}^2+\frac{|U_z|}{U^2}\psi_z^2\right)\nonumber\\
			&\leq\varliminf_{\epsilon\rightarrow0^{+} } \left[\int\left(\frac{\phi_z^2}{UU_\epsilon}+ \frac{\psi_z^2}{U_\epsilon} \right)+\int_{0}^{t}\int \left(\frac{U^{p-2}}{U_\epsilon}\phi_{zz}^2
			+\frac{|{U}_z|}{U_\epsilon^2}\psi_z^2\right) \right]\nonumber\\&\leq C\int\left(\frac{\phi_{0z}^2}{U^2}+ \frac{\psi_{0z}^2}{U}+\frac{\phi_{0}^2 }{U^{\alpha+1}}+\frac{\psi_{0}^2 }{U^{\alpha}} \right)+CN(t)\int_{0}^{t}\int U^{1-p}\psi_z^2.\label{eq18}	
		\end{align}
Next, we claim
	\begin{equation}\label{Upsiz}
		\int_{0}^{t}\int U\psi_{z}^2\leq C\left(\|\phi_0\|_{w_1}^2+\|\psi_0\|_{w_2}^2+\|\phi_{0z}\|_{w_3}^2
		+\|\psi_{0z}\|_{w_4}^2+N(t)\int_{0}^{t}\int U^{1-p}\psi_{z}^2\right).
	\end{equation}
	To prove \eqref{Upsiz}, we multiply the first equation of \eqref{phipsi} by $\psi_{z}$ to get
	\begin{equation}\label{eq46}
		\chi U \psi_z^2=\phi_t \psi_z-\left[ (U+\phi_{z})^p-U^p\right]_z \psi_z-\left(s+\chi V \right) \phi_z \psi_z-\chi \phi_z \psi_z^2.
	\end{equation}
	Noting that by the second equation of \eqref{phizpsiz}, it holds
	\begin{equation}\nonumber
		\begin{aligned}
			\phi_t \psi_z & =\left(\phi \psi_z\right)_t-\phi \psi_{z t}=\left(\phi \psi_z\right)_t-\phi\left(s \psi_{z z}+\phi_{z z}\right) \\
			& =\left(\phi \psi_z\right)_t-s\left(\phi \psi_z\right)_z+s \phi_z \psi_z-\left(\phi \phi_z\right)_z+\phi_z^2.
		\end{aligned}
	\end{equation}Integrating \eqref{eq46} over $[0,t]\times \mathbb{R}$,
	we obtain
		\begin{align}
			\chi\int_{0}^{t} \int U \psi_z^2&=\int \phi\psi_{z}-\int \phi_{0}\psi_{0z}+\int_{0}^{t}\int\phi_{z}^2-\int_{0}^{t}\int \left[ (U+\phi_{z})^p-U^p\right]_z \psi_z\nonumber\\&\quad-\chi\int_{0}^{t}\int V\phi_z \psi_z-\chi \int_{0}^{t}\int\phi_z \psi_z^2,\label{eq47} 		
		\end{align}
	where, in view of \eqref{|Uz|estimate} and the Taylor's expansion, it holds that
	\begin{equation}\nonumber
		\begin{aligned}
			&\quad-\int_{0}^{t}\int \left[ (U+\phi_{z})^p-U^p\right]_z \psi_z\\&=-p\int_{0}^{t}\int(U+\phi_{z})^{p-1}\phi_{zz}\psi_z-p\int_{0}^{t}\int \left[(U+\phi_{z})^{p-1}-U^{p-1}\right]U_z\psi_{z}\\&\leq C\left(\int_{0}^{t}\int U^{p-1}|\phi_{z z}||\psi_{z}|+\int_{0}^{t}\int U^{p-2}|U_z||\phi_{z}||\psi_{z}|\right)\\&\leq \frac{\chi}{4}\int_{0}^{t} \int U \psi_z^2+C\left(\int_{0}^{t}\int U^{p-3}\phi_{zz}^2+\int_{0}^{t}\int U^{p-2-\alpha}\phi_{z}^2 \right).
		\end{aligned}
	\end{equation}
	We thus have from \eqref{eq47} that
	\begin{equation}\nonumber
		\begin{aligned}
			\!\chi\!\int_{0}^{t} \int U \psi_z^2\!\leq\! C\int_{0}^{t}\int\left( U^{p-3}\phi_{zz}^2\!+\!U^{p-2-\alpha}\phi_{z}^2\right)\!+\!C\int\left( \frac{\phi^2}{U^{\alpha+1}}\!+\!\frac{\psi_{z}^2}{U}\!+\! \frac{\phi_0^2}{U^{\alpha+1}}\!+ \!\frac{\psi_{0z}^2}{U}\right),		
		\end{aligned}
	\end{equation}
	where we have used the fact $\left\|\psi_z\right\|_{L^{\infty}} \leq N(t)$ and $|V|\leq C$. This along with \eqref{L2 estimate} and \eqref{eq18} gives rise to \eqref{Upsiz}, provided that $N(t)$ is small enough.
	
	Now substituting \eqref{Upsiz} into \eqref{eq18} yields
		\begin{align}
			&\int\left(\frac{\phi_z^2}{U^2}+ \frac{\psi_z^2}{U} \right)+\int_{0}^{t}\int U^{p-3} \phi_{zz}^2+\int_{0}^{t}\int_{0}^{+\infty}\left(\frac{|U_z|}{U^2}-CN(t)U^{1-p}\right)\psi_z^2\nonumber
			\\&+\int_{0}^{t}\int_{-\infty}^{0}\left(U-CN(t)U^{1-p}\right)\psi_z^2
			+\int_{0}^{t}\int_{-\infty}^{0}\frac{|U_z|}{U^2}\psi_z^2
			+\int_{0}^{t}\int_{0}^{+\infty}U\psi_z^2\nonumber\\&\leq C\left(\int\frac{\phi_{0}^2}{U^{\alpha+1}}+\int\frac{\psi_{0}^2}{U^{\alpha}}
			+\int\frac{\phi_{0z}^2}{U^2}+\int\frac{\psi_{0z}^2}{U} \right).\label{eq48}
		\end{align}
	Recalling that $U$ is monotone decreasing, by \eqref{3.2}, we have for $z\in[0,+\infty)$,
	\[ \frac{|U_z|}{U^2}(z)=\frac{\chi}{sp}(u_--U)U^{-p}(z)\geq\frac{\chi}{sp}(u_--U(0))U^{-p}(0) \text{ and } U^{1-p}(z)\leq U^{1-p}(0).\]
	Thus, if $N(t)$ is small such that $CN(t)U^{1-p}(0)\leq\frac{\chi}{2s p}(u_--U(0))U^{-p}(0)$, then
	\[\frac{|U_z|}{U^2}-CN(t)U^{1-p}\geq\frac{|U_z|}{2U^2} \text{ for } z\in[0,+\infty).\]
	Similarly, since $U(z)>U(0)$ and $U^{1-p}(z)\leq u_-^{1-p}$ for $z\in(-\infty,0)$, if $N(t)\leq\frac{U(0)}{2Cu_-^{1-p}}$, then
	\[U-CN(t)U^{1-p}\geq\frac{U}{2} \text{ for }z\in(-\infty,0).\]
	It thus follows from \eqref{eq48} that, if $N(t)\ll1$, then
		\begin{align}
			&\int\left(\frac{\phi_z^2}{U^2}+ \frac{\psi_z^2}{U} \right)+\int_{0}^{t}\int U^{p-3} \phi_{zz}^2+\int_{0}^{t}\int\left(\frac{|U_z|}{U^2}\psi_z^2+U\psi_z^2\right)\nonumber\\&\leq C\left(\int\frac{\phi_{0}^2}{U^{\alpha+1}}+\int\frac{\psi_{0}^2}{U^{\alpha}}
			+\int\frac{\phi_{0z}^2}{U^2}+\int\frac{\psi_{0z}^2}{U} \right).\label{4.50}
		\end{align}
	The desired estimate \eqref{H1 estimate} follows from \eqref{4.50}.
\end{proof}

To close the \textit{a priori} estimate, we further need the estimate of the second order derivative of $(\phi,\psi)$.
\begin{lemma}\label{H2}
	Let the assumptions of Proposition \ref{proposition priori estimate} hold. Then it holds for any $t\in[0,T]$,
		\begin{align}
			&\|\phi_{zz}\|_{w_3}^2+\|\psi_{zz}\|_{w_4}^2+\int_0^t\left(\left\|\phi_{zzz}(\tau)\right\|_{w_6}^2+\left\|\psi_{zz}(\tau)\right\|^2 \right)\nonumber \\&\leq C\left(\|\phi_0\|_{w_1}^2+\|\psi_0\|_{w_2}^2+\|\phi_{0z}\|_{1,w_3}^2
			+\|\psi_{0z}\|_{1,w_4}^2\right).\label{H2 estimate}
		\end{align}

\end{lemma}

\begin{proof}
	We differentiate \eqref{phizpsiz} with respect to $z$ to get
	\begin{equation}\label{phizzpsizz}
		\begin{cases}
			\phi_{zz t}=&\left(pU^{p-1}\phi_{zzz}\right)_z+\chi U \psi_{zzz}+\chi\left( 2U_z \psi_{zz}+U_{zz} \psi_{z}\right)+\left[\left( s+\chi V\right)\phi_z \right]_{z z}\\&+\chi\left(\phi_z \psi_z\right)_{zz}+\left(2p(U^{p-1})_z\phi_{zz}+ p(U^{p-1})_{zz}\phi_{z}\right)_z+F_{zzz}, \\
			\psi_{zzt}=&s \psi_{z zz}+\phi_{zzz}.
		\end{cases}
	\end{equation}	
	Multiplying the first equation of \eqref{phizzpsizz} by $\frac{\phi_{zz}}{UU_\epsilon}$ and the second one by $\chi\frac{\psi_{zz}}{U_\epsilon}$, we obtain
		\begin{align}
			&\frac{1}{2}\frac{\mathrm{d}}{\mathrm{d}t}\int\left(\frac{\phi_{zz}^2}{UU_\epsilon}+\chi \frac{\psi_{zz}^2}{U_\epsilon} \right)+p\int \frac{U^{p-2}}{U_\epsilon} \phi_{zzz}^2-\frac{s\chi}{2}\int\frac{U_{z}}{U_\epsilon^2}\psi_{zz}^2\nonumber\\&=p\int\left(\frac{U^{p-3}U_z }{U_\epsilon}+\frac{U^{p-2}{U}_z }{U_\epsilon^2}\right) \phi_{zz}\phi_{zzz}+\chi\int\frac{{U}_z}{U_\epsilon^2}\phi_{zz}\psi_{zz}\nonumber
			\\&\quad+\chi\int\left( 2U_z \psi_{zz}+U_{zz} \psi_{z}\right)\frac{\phi_{zz}}{UU_\epsilon}+\int \left[\left( s+\chi V\right)\phi_z \right]_{z z}\frac{\phi_{zz}}{UU_\epsilon}+\chi\int\frac{\left(\phi_{z}\psi_{z} \right)_{zz}\phi_{zz}}{UU_\epsilon}\nonumber
			\\&\quad+\int \left(2p(U^{p-1})_z\phi_{zz}+ p(U^{p-1})_{zz}\phi_{z}\right)_z\frac{\phi_{zz}}{UU_\epsilon}
			+\int F_{zzz} \frac{\phi_{zz}}{UU_\epsilon}\nonumber\\&\triangleq I_1+\cdots+I_7.\label{eq19}
		\end{align}
	By \eqref{|Uz|estimate} and the Cauchy-Schwarz inequality, we have
		\begin{align}
			I_1&\leq C\left( \int \frac{|\phi_{zz}||\phi_{z zz}|}{UU_\epsilon}+\int\frac{|\phi_{zz}||\phi_{z zz}|}{U_\epsilon^2}\right)\nonumber\\&\leq \frac{p}{4}\int	\frac{U^{p-2} }{U_\epsilon}\phi_{zzz}^2+C\left( \int\frac{U^{2-p} }{U^{2}U_\epsilon}\phi_{zz}^2+\int\frac{U^{2-p} }{U_\epsilon^3}\phi_{zz}^2\right)\nonumber\\&\leq \frac{p}{4}\int	\frac{U^{p-2} }{U_\epsilon}\phi_{zzz}^2+C\int\frac{ \phi_{zz}^2}{U^{p+1}},	\label{eq22}	\end{align}
	and
	\begin{equation}\label{4.54}
		I_2\leq\frac{s\chi}{8}\int\frac{|{U}_z|}{U_\epsilon^2}\psi_{zz}^2+C\int\frac{\phi_{zz}^2}{U^{p}}.
	\end{equation}
	Notice that \eqref{3.2} gives
		\begin{align}
			U_{zz}&=\frac{\chi}{sp}U^{2-p}U_z+\frac{(2-p)}{sp}U^{1-p}U_z(\chi U-s^2)\nonumber\\&=\frac{\chi}{sp}U^{2-p}U_z+(2-p)\frac{U_z^2}{U},\label{eq20}
		\end{align}
	which in combination with \eqref{|Uz|estimate} implies that
	\begin{equation}\label{Uzz}
		|U_{zz}|\leq C\left(U^{4-2p}+U^{3-2p} \right)\leq CU^{3-2p}.
	\end{equation}
	Owing to the second equation of \eqref{travleing wave equ}, it holds
	\begin{equation}\label{eq21}
		|V_{zz}|=\frac{|U_{zz}|}{s}\leq CU^{3-2p}\leq C.
	\end{equation}
	On the other hand, \eqref{eq20} also gives
	\begin{equation}\label{4.58}
		(U^{p-1})_{zz}=(p-1)\left[(p-2)U^{p-3}U_z^2+U^{p-2}U_{zz} \right]=\frac{\chi(p-1)}{sp}U_z.
	\end{equation}
	Now by \eqref{|Uz|estimate} and \eqref{Uzz}, we get
		\begin{align}
			I_3&\leq2\chi\int\frac{|U_z|}{UU_\epsilon}|\phi_{zz}||\psi_{zz}|+C\int U^{1-2p}|\phi_{zz}||\psi_{z}|\nonumber\\&\leq \frac{s\chi}{8}\int\frac{|U_z|}{U_\epsilon^2}\psi_{zz}^2+C\left(\int U^{p-3} \phi_{zz}^2+\int \psi_{z}^2\right).\label{4.59}
		\end{align}
	By \eqref{travleing wave equ}, \eqref{eq21} and the fact that $|s+\chi V|\leq C$, we have
		\begin{align}
			I_4&=\int \frac{s+\chi V}{UU_\epsilon}\phi_{zz}\phi_{zzz}+2\chi \int \frac{V_z}{UU_\epsilon}\phi_{zz}^2+\chi\int \frac{V_{zz}}{UU_\epsilon}\phi_{z}\phi_{zz}\nonumber\\&\leq\frac{p}{8}\int\frac{ U^{p-2}}{U_\epsilon} \phi_{z zz}^2+C\left(\int \frac{\phi_{zz}^2}{U^{p}U_\epsilon}+\int \frac{\phi_{zz}^2}{U^{p}}+\int U^{p-3}\phi_{zz}^2+ \int \frac{U^{3-p}}{U^2U_\epsilon^2}\phi_{z}^2\right)\nonumber\\&\leq\frac{p}{8}\int \frac{ U^{p-2}}{U_\epsilon}\phi_{z zz}^2+C\left(\int U^{p-3}  \phi_{zz}^2+\int\frac{\phi_{z}^2 }{U^{p+1}} \right),\label{4.60}
		\end{align}
where we have used $\frac{1}{ U^{p}}\leq \frac{C}{ U^{p+1}}\leq CU^{p-3}$ in the second inequality.
	Noting $\|\psi_{z}\|_{L^{\infty}}\leq CN(t)$ and $\|\phi_{z}/U\|_{L^{\infty}}\leq CN(t)$, thanks to \eqref{|Uz|estimate}, one obtains
		\begin{align} I_5&=-\chi\int\frac{\phi_{zz}\phi_{zzz}\psi_z}{UU_\epsilon}
			+\chi\int\frac{U_z}{U^2U_\epsilon}\phi_{zz}^2\psi_z+\chi\int\frac{U_{z}}{UU_\epsilon^2}\phi_{zz}^2\psi_z\nonumber\\
			&\quad-\chi\int\frac{\phi_z\psi_{zz}\phi_{zzz}}{UU_\epsilon}
			+\chi\int\frac{U_z}{U^2U_\epsilon}\phi_z\psi_{zz}\phi_{zz}
			+\chi\int\frac{{U}_z}{U_\epsilon^2U}\phi_z\psi_{zz}\phi_{zz}\nonumber\\&\leq CN(t)\int\left(\frac{\phi_{zz}^2 }{U^pU_\epsilon}+\frac{U^{2-p}}{U_\epsilon} \psi_{zz}^2+\frac{|{U_z|}}{U^2}\phi_{zz}^2
			+\frac{|{U}_z|}{U_\epsilon^2}\phi_{zz}^2+\frac{U^{p-2} }{U_\epsilon}\phi_{z zz}^2+\frac{|{U}_z|}{U_\epsilon^2}\psi_{zz}^2\right)
			\nonumber\\&\leq\! CN(t)\int \left(\frac{U^{p\!-\!2} }{U_\epsilon}\phi_{z zz}^2\!+\!\frac{|{U}_z|}{U_\epsilon^2}\psi_{zz}^2+U^{1-p}\psi_{zz}^2\right)\!+\! C\int\left(\frac{1}{U^{p+1}}\!+\!\frac{1}{U^{p}} \right) \phi_{zz}^2.\label{4.61}
		\end{align}
	The sixth term on the right hand side of \eqref{eq19} can be estimated as follows. After integration by parts, we get
		\begin{align}
			I_6&=-2p\int\frac{(U^{p-1})_z}{UU_\epsilon}\phi_{zz}\phi_{zzz}
			+2p\int\left[\frac{(U^{p-1})_zU_z}{U^2U_\epsilon}+\frac{(U^{p-1})_zU_{z}}{UU_\epsilon^2}\right]\phi_{zz}^2
			\nonumber\\&\quad-p\int\frac{(U^{p-1})_{zz}}{UU_\epsilon}\phi_{z}\phi_{zzz}
			+p\int\left[\frac{(U^{p-1})_{zz}U_z}{U^2U_\epsilon}+\frac{(U^{p-1})_{zz}U_{z}}{UU_\epsilon^2}\right]\phi_{z}\phi_{zz}\nonumber\\&\triangleq I_{61}+\cdots+I_{64}\nonumber,
		\end{align}
	where, by virtue of \eqref{|Uz|estimate} and the Cauchy-Schwarz inequality, it holds that
	\begin{equation}\nonumber
		I_{61}=-2p(p-1)\int\frac{U^{p-3}U_{z} }{U_\epsilon} \phi_{zz}\phi_{zzz}\leq\frac{p}{8}\int \frac{U^{p-2} }{U_\epsilon}\phi_{z zz}^2+C\int\frac{\phi_{zz}^2 }{U^{p+1}} ,
	\end{equation}
	and
	\begin{equation}\nonumber
		I_{62}=2p(p-1)\int\left(\frac{U^{p-4} U_z^2 }{U_\epsilon} +\frac{U^{p-3} U_z^2 }{U_\epsilon^2}\right) \phi_{zz}^2\leq C\int \frac{\phi_{zz}^2 }{U^{p+1}}.	 	
	\end{equation}
	Thanks to \eqref{4.58} and \eqref{|Uz|estimate}, we have
	\begin{equation}\nonumber
		I_{63}=-\frac{\chi(p-1)}{s}\int\frac{U_z}{UU_\epsilon}\phi_{z}\phi_{zzz} \leq\frac{p}{8}\int\frac{U^{p-2}}{U_\epsilon} \phi_{z zz}^2+C\int U^{3-3p}\phi_{z}^2,
	\end{equation}
	and
	\begin{equation}\nonumber
		\begin{aligned}
			I_{64}=\frac{\chi(p-1)}{s}\int\left(\frac{U_z^2}{U^2U_\epsilon} +\frac{U_z^2}{UU^2_\epsilon}\right) \phi_{z}\phi_{zz}\leq C\left(\int U^{p-3}\phi_{zz}^2+\int U^{5-5p}\phi_{z}^2\right).
		\end{aligned}
	\end{equation}
	Thus, noting $0<p<1$, we have
	\begin{equation}\label{4.62}
		I_6\leq \frac{p}{4}\int\frac{U^{p-2}}{U_\epsilon} \phi_{z zz}^2+C\int U^{p-3}\phi_{zz}^2+C\int \phi_{z}^2.
	\end{equation}
	We next estimate the last term on the right hand side of \eqref{eq19}. By \eqref{Fzz}, we get
	\begin{equation}\nonumber
		\begin{aligned}
			I_7&=-\int F_{z z}\left(\frac{\phi_{zzz}}{UU_\epsilon}-\frac{U_z\phi_{zz}}{U^2U_\epsilon}
			-\frac{{U}_z\phi_{zz}}{U_\epsilon^2U}\right)\\&\leq C\int\left[\frac{U^{p-3}}{U_\epsilon} \phi_{z z}^2 |\phi_{z z z}|+\left(\frac{U^{p-4}}{U_\epsilon} |U_z| +\frac{U^{p-3}}{U_\epsilon^2} |U_{z}|\right) |\phi_{z z}^3|\right]\\
			& \quad+C\int\left[\left(\frac{U^{p-5}}{U_\epsilon} U_z^2+\frac{|U_z|}{U^2U_\epsilon}\right)\phi_z^2 |\phi_{z z z}|+\left (\frac{U^{p-6}|U_z^3|}{U_\epsilon} +\frac{ U_z^2}{U^3U_\epsilon}\right) \phi_z^2 |\phi_{z z}|\right.\\
			& \left. \quad+\left(\frac{U^{p-5}|U_z^3|}{U_\epsilon^2} +\frac{ U_z^2}{U^2U_\epsilon^2}\right) \phi_z^2 |\phi_{z z}| +\left(\frac{U^{p-5} U_z^2 }{U_\epsilon}+\frac{U^{p-4} |U_z|^2 }{U_\epsilon^2} \right)  |\phi_z| \phi_{z z}^2\right.\\
			& \left. \quad+\left(\frac{U^{p-4}|U_z|}{U_\epsilon}+\frac{U^{p-3}|U_{z}|}{U_\epsilon^2} \right)   |\phi_z| |\phi_{z z}|| \phi_{z z z}|+\frac{U^{p-3} }{U_\epsilon} |\phi_z| \phi_{z z z}^2\right]\\&
			\triangleq I_{71}+I_{72}.
		\end{aligned}
	\end{equation}
	By \eqref{|Uz|estimate}, $\|\phi_{zz}\|_{w_3}\leq N(t)$ and H\"{o}lder's inequality, we derive
	\begin{equation}\nonumber
		\begin{aligned}
			I_{71}&\leq\|U^{\frac{p-2}{2}}U_\epsilon^{-\frac{1}{2}}\phi_{z z}\|_{L^{\infty}} \|\frac{\phi_{zz} }{U}\|_{L^{2}}\| U^{\frac{p-2 }{2}}U_\epsilon^{-\frac{1}{2}}\phi_{zzz}\|_{L^{2}}\\
			&\!\quad+\|U^{\frac{p\!-\!2}{2}}U_\epsilon^{\!-\!\frac{1}{2}}\phi_{z z}\|_{L^{\infty}} \left( \| U^{\frac{\!\!-2p\!+\!1}{2}} U_\epsilon^{\!-\!\frac{1}{2}}\phi_{zz}\|_{L^{2}}\!+\! \| U^{\frac{3\!-\!2p }{2}} U_\epsilon^{\!-\!\frac{3}{2}}\phi_{zz}\|_{L^{2}} \right)  \| U^{\frac{p\!-\!3 }{2}}\phi_{zz}\|_{L^{2}}\\&\!\leq\!\! CN(t)\left(\|U^{\frac{p\!-\!2}{2}}U_\epsilon^{\!-\!\frac{1}{2}}\phi_{z z}\|_{L^{\infty}} \!\| U^{\frac{p\!-\!2 }{2}}U_\epsilon^{\!-\!\frac{1}{2}}\phi_{zzz}\|_{L^{2}}
			\!+\!\|U^{\frac{p\!-\!2}{2}}U_\epsilon^{\!-\!\frac{1}{2}}\phi_{z z}\|_{L^{\infty}} \!\| U^{\frac{p\!-\!3 }{2}}\phi_{zz}\!\|_{L^{2}} \!\right)\\&\leq \!CN(t)\left( \|U^{\frac{p\!-\!2}{2}}U_\epsilon^{\!-\!\frac{1}{2}}\phi_{z z}\|_{L^{2}}^\frac{1}{2}\|U^{\frac{p\!-\!2 }{2}}U_\epsilon^{\!-\!\frac{1}{2}}\phi_{z zz}\|_{L^{2}}^\frac{3}{2}\!+\!\|U^{\frac{\!-\!3}{2}}\phi_{z z}\|_{L^{2}}^\frac{3}{2}\|U^{\frac{p\!-\!2 }{2}}U_\epsilon^{\!-\!\frac{1}{2}}\phi_{z zz}\|_{L^{2}}^\frac{1}{2}\right)\\&\leq N(t)p\int \frac{U^{p-2} }{U_\epsilon}\phi_{zz z}^{2}+C\int U^{p-3}\phi_{zz}^{2},
		\end{aligned}
	\end{equation}
	where we have used in the second inequality that $\frac{1}{U^{2p}}\leq \frac{C}{U^2}$ due to $0<p<1$. And thanks to \eqref{|Uz|estimate} and $\|\phi_{z}/U\|_{L^{\infty}}\leq CN(t)$, it holds that
	\begin{equation*}
		\begin{aligned}
			I_{72}&\leq CN(t)\int\left(\frac{|\phi_{z}||\phi_{z z z}|}{U^{p}U_\epsilon}+ \frac{|\phi_{z}||\phi_{z z}|}{U^{2p}}+\frac{|\phi_{zz}||\phi_{z z z}|}{UU_\epsilon}+\frac{\phi_{z z}^2}{U^{p+1}}+p\frac{ U^{p-2}}{U_\epsilon}\phi_{z zz}^2\right)\\&\leq\! N(t)p\int\frac{U^{p\!-\!2} }{U_\epsilon} \phi_{z zz}^2\!+\!C\left[\int\left(\frac{ U^{2\!-\!3p}}{U_\epsilon}\!+\!\frac{1}{U^{2p}}\right) \phi_{z}^2\!+\!\int\left(\frac{1}{U^{2p} }\!+\!\frac{U^{2\!-\!p}}{U^2U_\epsilon}+\frac{1}{U^{p+1}}\right) \phi_{zz}^2 \right]\\&\leq 	 N(t)p\int \frac{ U^{p-2}}{U_\epsilon}\phi_{z zz}^2+C\left(\int \frac{\phi_{z}^2}{U^{p+1}} +\int \frac{\phi_{zz}^2}{U^{p+1}} \right),	
		\end{aligned}
	\end{equation*}
	where we have used the fact $U^{1-3p}\leq \frac{C}{U^{2p}}\leq \frac{C}{U^{p+1}}$ in the last inequality. Thus,
	\begin{equation}\label{eq27}
		I_{7}\leq 	CN(t)p\int \frac{ U^{p-2}}{U_\epsilon}\phi_{z zz}^2+C\int U^{p-3}\phi_{zz}^2+C\int \frac{\phi_{z}^2}{U^{p+1}}.
	\end{equation}
	Now substituting \eqref{eq22}-\eqref{eq27} into \eqref{eq19}, and integrating the inequality with respect to $t$, by Lemma \ref{H1}, we obtain
		\begin{align}
			&\quad\int\left(\frac{\phi_{zz}^2}{UU_\epsilon}+\frac{\psi_z^2}{U_\epsilon} \right)+\int_{0}^{t}\left(\int \frac{U^{p-2} }{U_\epsilon} \phi_{zzz}^2+\int\frac{|U_{z}|}{U_\epsilon^2}\psi_{zz}^2 \right)\nonumber\\&\leq\!\int\left(\frac{\phi_{0zz}^2}{UU_\epsilon}\!+\!\frac{\psi_{0zz}^2}{U_\epsilon} \right)\!+\!C\int_{0}^{t} \int \left(U^{p-3}\phi_{zz}^2+ U^{p\!-\!2\!-\!\alpha} \phi_{z}^2\!+\!\psi_{z}^2 \right)\!+\! CN(t)\int_{0}^{t} \int U^{1\!-\!p}\psi_{zz}^2\nonumber\\&\leq C\int\left(\frac{\phi_{0zz}^2}{U^2}+\frac{\psi_{0zz}^2}{U}+\frac{\phi_{0z}^2}{U^2}+\frac{\psi_{0z}^2}{U}+\frac{\phi_{0}^2}{U^{\alpha+1} }+\frac{\psi_{0}^2}{U^{\alpha} } \right)+CN(t)\int_{0}^{t} \int U^{1-p}\psi_{zz}^2,\label{eq28} \end{align}	
	provided that $N(t)$ is suitably small. Letting $\varepsilon\rightarrow0^{+}$, by Fatou's Lemma, it then follows that
		\begin{align}
			&\quad\int\left(\frac{\phi_{zz}^2}{U^2}+\frac{\psi_z^2}{U} \right)+\int_{0}^{t}\left(\int U^{p-3} \phi_{zzz}^2+\int\frac{|U_{ z}|}{U^2}\psi_{zz}^2 \right)\nonumber\\&\leq C\int\left(\frac{\phi_{0zz}^2}{U^2}+\frac{\psi_{0zz}^2}{U}+\frac{\phi_{0z}^2}{U^2}+\frac{\psi_{0z}^2}{U}+\frac{\phi_{0}^2}{U^{\alpha+1} }+\frac{\psi_{0}^2}{U^{\alpha} } \right)+CN(t)\int_{0}^{t} \int U^{1-p}\psi_{zz}^2.\label{eq52}
		\end{align}

	Next we estimate $\int_{0}^{t} \int U^{1-p}\psi_{zz}^2$. Multiplying the first equation of \eqref{phizpsiz} by $\psi_{zz}$, we get
	\begin{equation}\nonumber
		\begin{aligned}
			\chi U \psi_{z z}^2=&\phi_{z t} \psi_{z z}-\left[(U+\phi_{z})^p-U^p\right]_{zz} \psi_{z z}-\left(s +\chi V \right) \phi_{z z}\psi_{z z}\\&-\chi V_z \phi_z\psi_{z z}-\chi U_z \psi_z\psi_{z z}-\chi\left(\phi_z \psi_z\right)_z\psi_{z z}.		
		\end{aligned}
	\end{equation}
	Noting that
	\begin{equation}\nonumber
		\begin{aligned}
			\phi_{z t} \psi_{z z} & =\left(\phi_z \psi_{z z}\right)_t-\phi_z \psi_{z z t}=\left(\phi_z \psi_{z z}\right)_t-\phi_z\left(s \psi_{z z z}+\phi_{z z z}\right) \\
			& =\left(\phi_z \psi_{z z}\right)_t-s\left(\phi_z \psi_{z z}\right)_z+s \phi_{z z} \psi_{z z}-\left(\phi_z \phi_{z z}\right)_z+\phi_{z z}^2,
		\end{aligned}
	\end{equation}
	we have
		\begin{align}
			\chi \int_{0}^{t}\int U \psi_{z z}^2=&\int\phi_z \psi_{zz}-\int\phi_{0z} \psi_{0z z}+\int_{0}^{t}\int\phi_{z z}^2-\int_{0}^{t}\int\left[(U+\phi_{z})^p-U^p\right]_{zz} \psi_{z z}\nonumber\\&-\chi\int_{0}^{t}\int (V \phi_{z z}+V_z \phi_z+U_z \psi_z)\psi_{z z}-\chi\int_{0}^{t}\int\left(\phi_z \psi_z\right)_z\psi_{z z},\label{eq49}
		\end{align}
	where,
		\begin{align}
			&-\int_{0}^{t}\int\left[(U+\phi_{z})^p-U^p\right]_{zz} \psi_{z z}\nonumber\\=&  -\int_{0}^{t} \int \left[p(p-1)\left(U+\phi_z\right)^{p-2}\left(U_z+\phi_{z z}\right)^2+p\left(U+\phi_z\right)^{p-1}\left(U_{z z}+\phi_{z z z}\right)\right.\nonumber\\&\quad\left.-p(p-1) U^{p-2} U_z^2-p U^{p-1} U_{z z}\right] \psi_{z z} \nonumber\\
			=&  \!-\!p(p\!-\!1)\int_{0}^{t} \int \left(U\!+\!\phi_z\right)^{p\!-\!2} \phi_{z z}^2 \psi_{z z}\!\!p(p\!-\!1)\int_{0}^{t} \int \left[\left(U\!\!\phi_z\right)^{p\!-\!2}\!-\!U^{p\!-\!2}\right] U_z^2 \psi_{z z}\nonumber\\&\!-\!2p(p\!-\!1)\int_{0}^{t} \int  \left(U+\phi_z\right)^{p-2} U_z \phi_{z z} \psi_{z z}-p\int_{0}^{t} \int\left[ \left(U+\phi_z\right)^{p-1}-U^{p-1}\right] U_{z z} \psi_{z z} \nonumber\\&-p\int_{0}^{t} \int \left(U+\phi_z\right)^{p-1} \phi_{z z z} \psi_{z z}.\label{eq50}
		\end{align}
	By Taylor's expansion, $\|\psi_{zz}\|_{w_4}\leq N(t)$ and H\"{o}lder's inequality, we deduce that the first term on the right hand side of \eqref{eq50} satisfies
		\begin{align}
			\left|\int_{0}^{t} \int \left(U+\phi_z\right)^{p-2} \phi_{z z}^2 \psi_{z z}\right|&\leq C\int_{0}^{t}\int U^{{p-2}}\phi_{z z}^2 |\psi_{zz}| 	\nonumber\\&\leq C\int_{0}^{t}\|U^{\frac{p-3}{2}}\phi_{z z}\|_{L^{\infty}} \|U^{\frac{p}{2}}\phi_{z z}\|_{L^{2}}\|U^{-\frac{1}{2}}\psi_{zz}\|_{L^{2}}\nonumber\\&\leq CN(t)\int_{0}^{t}\|U^{\frac{p-3}{2}}\phi_{zzz}\|_{L^{2}}^{\frac{1}{2}} \|U^{\frac{p-3}{2}}\phi_{z z}\|_{L^{2}}^{\frac{3}{2}}\nonumber\\&\leq C\left(\int_{0}^{t}\int U^{p-3}\phi_{zzz}^2+\int_{0}^{t}\int U^{p-3}\phi_{zz}^2 \right).\label{4.68}
		\end{align}
	Utilizing \eqref{|Uz|estimate}, \eqref{Uzz} and the Cauchy-Schwarz inequality, we have
		\begin{align}
			&\quad\text{the other terms of \eqref{eq50}}\nonumber\\&\leq C\left(\int_{0}^{t}\int U^{1-p} |\phi_{z}||\psi_{zz}|+\int_{0}^{t}\int  |\phi_{zz}||\psi_{zz}|+\int_{0}^{t}\int U^{p-1} |\phi_{zzz}||\psi_{zz}| \right)\nonumber	\\&\leq \frac{3\chi}{8}	\int_{0}^{t}\int U \psi_{zz}^2+C\left(\int_{0}^{t}\int U^{2p-3}\phi_{zzz}^2+\int_{0}^{t}\int \frac{\phi_{zz}^2}{U} +\int_{0}^{t}\int U^{1-2p}\phi_{z}^2 \right).\label{eq4.69}
		\end{align}
	Moreover, since $|V|\leq C$ and $V_z=-\frac{U_z}{s}$, by \eqref{|Uz|estimate}, we get
		\begin{align}
			&\left|\chi\int_{0}^{t}\int (V \phi_{z z}+V_z \phi_z+U_z \psi_z)\psi_{z z}\right|\nonumber\\ \leq& \frac{3\chi}{8}	\int_{0}^{t}\int U \psi_{zz}^2 +C\left(\int_{0}^{t}\int U^{p-3}\phi_{zz}^2+\int_{0}^{t}\int U^{p-\alpha-2}\phi_{z}^2+\int_{0}^{t}\int\psi_{z}^2 \right).\label{4.70}
		\end{align}
	The last term on the right hand side of \eqref{eq49} can be estimated as
		\begin{align}
			-\chi\int_{0}^{t}\int\left(\phi_z \psi_z\right)_z\psi_{z z}&=-\chi \int_{0}^{t}\int \phi_{zz} \psi_{z}\psi_{zz}-\chi \int_{0}^{t}\int \phi_{z} \psi_{zz}^2\nonumber\\&\leq  CN(t)\chi \int_{0}^{t}\int U \psi_{zz}^2+ C\int_{0}^{t}\int U^{p-3}\phi_{zz}^2,	\label{eq51}
		\end{align}
where we have used $\|\phi_{z}/U\|_{L^{\infty}}\leq CN(t)$ and $\|\psi_{z}\|_{L^{\infty}}\leq CN(t)$.
	Because $1/U\leq CU^{2p-3}\leq C U^{p-3}$ and $U^{1-2p}\leq CU^{p-\alpha-2}$, adding \eqref{eq50}-\eqref{eq51} with \eqref{eq49} and utilizing \eqref{H1 estimate} and \eqref{eq52}, we then arrive at
		\begin{align}
			\chi \int_{0}^{t}\int U \psi_{z z}^2&\leq C\int_{0}^{t}\int \left(U^{p-3}\phi_{zzz}^2+U^{p-3}\phi_{zz}^2+ U^{p-\alpha-2}\phi_{z}^2+\psi_{z}^2 \right)\nonumber\\&\quad+C\int \left( \frac{\phi_{0z}^2}{U^2}+ \frac{\psi_{0zz}^2}{U}+\frac{\phi_{z}^2}{U^2}+\frac{\psi_{zz}^2}{U}\right)\nonumber\\&\leq 	 C\int\left(\frac{\phi_{0zz}^2}{U^2}+\frac{\psi_{0zz}^2}{U}+\frac{\phi_{0z}^2}{U^2}+\frac{\psi_{0z}^2}{U}+\frac{\phi_{0}^2}{U^{\alpha+1} }+\frac{\psi_{0}^2}{U^{\alpha} } \right)\nonumber\\&\quad+CN(t)\int_{0}^{t} \int U^{1-p}\psi_{zz}^2.\label{eq53}	
		\end{align}
Combining \eqref{eq52} and \eqref{eq53} and following the same procedure as in \eqref{eq48}, we obtain \eqref{H2 estimate} provided that $N(t)$ is small enough. Thus, the proof of Lemma \ref{H2} is finished.
\end{proof}

\begin{proof}[Proof of Proposition \ref{proposition priori estimate}] The desired \emph{a priori} estimate \eqref{priori estimate} follows from Lemmas \ref{L2}-\ref{H2}.
\end{proof}

\subsection{Proof of main results}

\begin{proof}[Proof of Theorem \ref{transformed stability}] The \emph{a priori} estimate \eqref{priori estimate} guarantees that $N(t)$ is small for all $t>0$ if $N(0)$ is small enough. Hence applying the standard extension procedure, we obtain the global well-posedness of the system  \eqref{phipsi}-\eqref{phi0psi0} in $X(0,\infty)$. Thanks to the decomposition \eqref{decompose}, the system \eqref{transformed model}-\eqref{transinitial data} has a unique global solution $(u,v)(x,t)$ satisfying \eqref{2.10}. It remains to show the convergence \eqref{2.12}. In view of \eqref{decompose} again, it suffices to show
	\begin{equation}\label{asymptotic}
		\sup _{z \in \mathbb{R}}|(\phi_{z}, \psi_z)(z, t)| \rightarrow 0, \quad \text { as } t \rightarrow+\infty,
	\end{equation}	
	where $z=x-st$. We denote $p(t)\triangleq\|(\phi_z,\psi_z)(\cdot,t)\|^2$. By the estimate \eqref{priori estimate}, we have
	$$
	\int_{0}^{\infty}\|(\phi_z,\psi_z)(\cdot,t)\|^2 \mathrm{d}t
	\leq \int_{0}^{\infty}\left(\left\|\phi_z(\cdot,t)\right\|_{w_5}^2
	+\left\|\psi_z(\cdot,t)\right\|^2\right)\mathrm{d}t\leq C N^2(0) \leq C,
	$$	
	which implies $p(t)\in L^1(0,\infty)$. We next show that $p'(t)\in L^1(0,\infty)$. By the equation  \eqref{phizpsiz}, and using $\|\psi_{z}\|_{L^{\infty}}, \|\phi_{z}\|_{L^{\infty}}\leq CN(t)$, we have
		\begin{align}\int_{0}^{\infty}\int\phi_{zt}^2&\leq C\int_{0}^{\infty}\int(U^{2p-2}\phi_{zzz}^2+U^{2p-4}U_z^2\phi_{zz}^2
			+U^{2p-6}U_z^4\phi_{z}^2+U^{2p-4}U_{zz}^4\phi_{z}^2)
			\nonumber\\&\quad+C\int_{0}^{\infty}\int(\phi_{zz}^2+\phi_{z}^2+\psi_{zz}^2+\psi_{z}^2
			+F_{zz}^2)\nonumber\\&\leq C\nonumber,\end{align}	
	and
	\[\int_{0}^{\infty}\int\psi_{zt}^2\leq C\int_{0}^{\infty}\int(\psi_{zz}^2+\phi_{zz}^2)\leq C.\]
	Thus,
	\begin{equation}\nonumber
		\begin{aligned}
			\int_{0}^{\infty } |p'(\tau)|\mathrm{d}\tau
			&=2 \int_{0}^{\infty } \left|\int\left(\phi_z\phi_{zt}+\psi_z\psi_{zt}
			\right)\right|\mathrm{d}\tau\\& \leq C\int_{0}^{\infty }\int \left( \phi_{z}^2+\phi_{z t}^2+\psi_{z}^2+\psi_{z t}^2\right)\mathrm{d}\tau\\&\leq C.
		\end{aligned}
	\end{equation}
	It then follows that
	\begin{equation}\nonumber
		\|(\phi_z,\psi_z)(\cdot,t)\| \rightarrow 0, \quad \text {as } t \rightarrow +\infty.
	\end{equation}
	Now for all $z\in \mathbb{R}$, we get
	\begin{equation}\nonumber
		\begin{aligned}
			\phi_z^2(z, t)  \!=\!2 \int_{\!-\!\infty}^z \phi_z \phi_{z z}(y, t) d y\!\leq \!2\|\phi_z(\cdot,t)\!\|\!\|\phi_{z z}(\cdot,t)\|\!\leq \!C\left\|\phi_z(\cdot, t)\right\| \rightarrow 0, \quad \text {as } t \rightarrow+\infty,
		\end{aligned}
	\end{equation}
	where we have used the boundedness of $\|\phi_{z z}\|$. Hence
	\[	\sup _{z \in \mathbb{R}}|\phi_{z}(z, t)| \rightarrow 0, \quad \text {as } t \rightarrow+\infty.
	\]
	Applying the same procedure to $\psi_{z}$ leads to
	\begin{equation}\nonumber
		\sup _{z \in \mathbb{R}}|\psi_{z}(z,t)|\rightarrow0, \quad \text {as } t \rightarrow+\infty.	
	\end{equation}
	Therefore \eqref{asymptotic} is proved and the proof of Theorem \ref{transformed stability} is complete.
\end{proof}

\begin{proof}[Proof of Theorem \ref{original TW stability}] Owing to \eqref{decompose} and the Hopf-Cole transformation \eqref{hopf-cole}, we have
	\[(\ln w_0(x)-\ln W(x+x_0))_x=-(v_0(x)-V(x+x_0))=-\psi_{0x}(x),\]
	which gives
	\[\ln w_0(x)-\ln W(x+x_0)=-\psi_{0}(x).\]
	Then it is easy to see that the assumption of Theorem \ref{original TW stability} verifies that of Theorem \ref{transformed stability}. Hence according to Theorem \ref{transformed stability},  the transformed problem \eqref{transformed model}-\eqref{transinitial data} admits a unique global solution $(u,v)(x,t)$ satisfying the regularity \eqref{2.10} and the asymptotic convergence \eqref{2.12}.
	
	We next show the desired results of $w$. By the second equation of \eqref{orginal model}, one can solve $w$ as
	\begin{equation*}
		w(x,t)=w_0(x)e^{-\int_0^tu(x,\tau)d\tau},
	\end{equation*}
	which, in combination with  the fact that $u(\infty,t)=0$, leads to
	\begin{equation}\label{eq54}
		w(\infty,t)=w_0(\infty)=w_+,\quad \forall \ t>0.	
	\end{equation}
	Now by \eqref{decompose} and the transformation \eqref{hopf-cole} again, one deduces that
	\begin{equation}\nonumber
		(\ln w(x,t)-\ln W(x-st+x_0))_x=-(v(x,t)-V(x-st+x_0))=-\psi_x(x,t).\end{equation}
	Hence
	$$w(x,t)=W(x-st+x_0)e^{-\psi(x,t)}.$$
	By Taylor's expansion, we have
		\begin{align}
			w(x,t)-W(x-st+x_0)&=W(x-st+x_0)(e^{-\psi(x,t)}-1)\nonumber\\&
			=W(x-st+x_0)\psi(x,t)\sum_{n=1}^{\infty}\frac{(-1)^{n}}{n!}(\psi)^{n-1}(x,t).\label{4.69}
	\end{align}
By Proposition \ref{phipsi stability}, one can see that $\|\psi\|_{C([0,\infty); H^2)}$ is small. Thus, we conclude that the series $\sum_{n=1}^{\infty}\frac{(-1)^{n}}{n!}(\psi)^{n-1}$ is absolutely convergent in $C([0,\infty); H^2)$. This together with the regularities of $\psi$ and $v$ implies
	\[w-W\in C([0,\infty);H^2),\ w_x/w-W_x/W \in C([0, \infty) ; H_{w_4}^1).\]
	It remains to show the asymptotic convergence of $w$. Thanks to \eqref{eq54}-\eqref{4.69} and the convergence of $\psi$, we have
	\begin{eqnarray*}
		\begin{aligned}
			\sup\limits_{x\in\mathbb{R}}|w(x,t)-W(x-st+x_0)|
			&=\sup\limits_{x\in\mathbb{R}}W(x-st+x_0)|e^{-\psi(x,t)}-1|\\&\leq w_+\sup\limits_{x\in\mathbb{R}}|e^{-\psi(x,t)}-1|\\&\to 0,  \ \ \mathrm{as} \ \ t \to \infty.
		\end{aligned}
	\end{eqnarray*}
	The proof is complete.
\end{proof}

\section{Numerical simulations}\label{numerical}
In this section, we are going to present some numerical simulations in three cases for the equations
\begin{equation}\label{numer1}
	\left\{\begin{aligned}
		u_t&=(u^p)_{x x}+(u v)_x, \\
		v_t&=u_x,
	\end{aligned}\right.
\end{equation}
with $p=0.1,0.5$ and $0.9$,  and the following initial data
\[
(u, v)(x, 0)=(u_0, v_0)(x) =\left\{\begin{aligned}
	&\Big(\frac{1}{2} e^{-x}, -\frac{1}{2} e^{-x}\Big), & \text { for } x \ge 0, \\
	&\Big(1-\frac{1}{2} e^{x}\cos x,-1+\frac{1}{2} e^{x}\cos x \Big), & \text { for } x <0.
\end{aligned}\right.
\]

\begin{figure}[H]
	\begin{minipage}[h]{0.48\linewidth}
		\centering
		\includegraphics[width=2.46in]{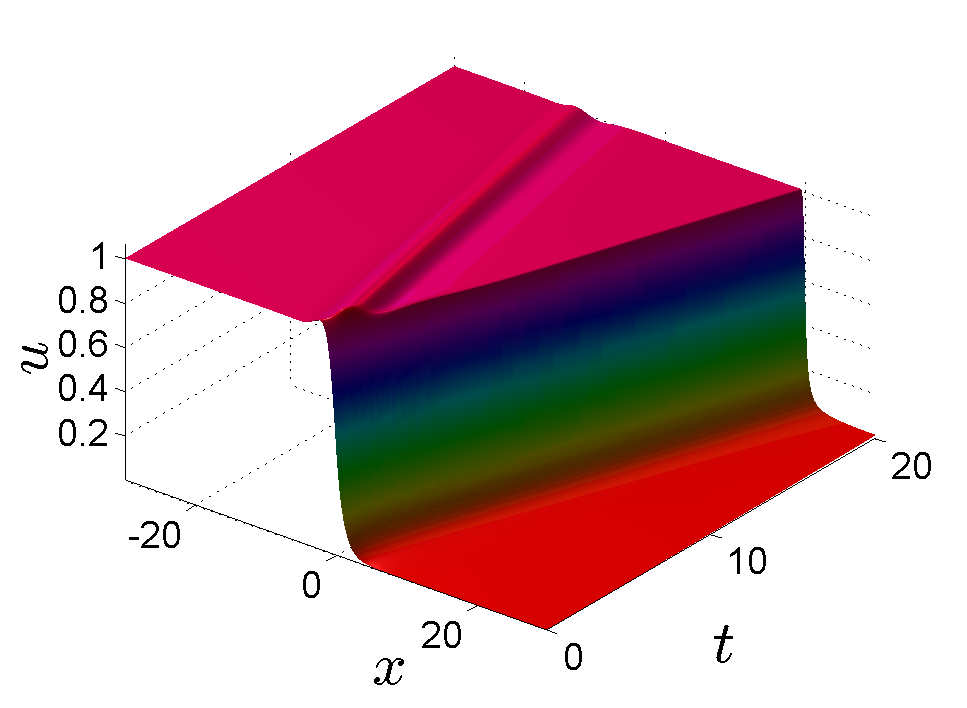}
	\end{minipage}%
	\begin{minipage}[h]{0.48\linewidth}
		\centering
		\includegraphics[width=2.46in]{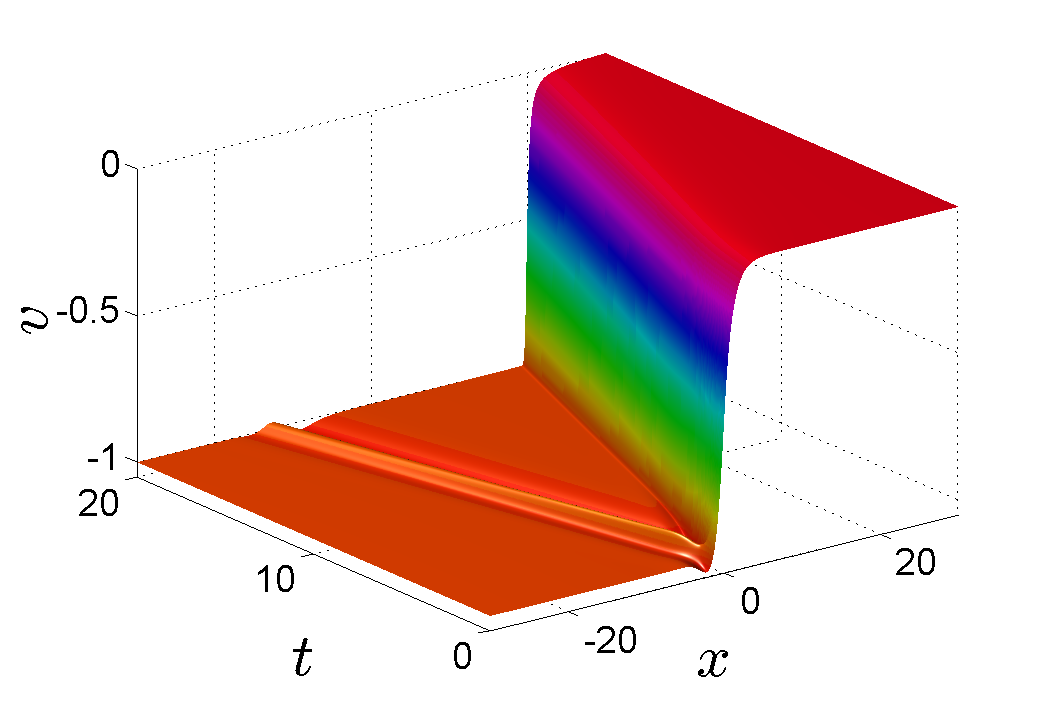}
	\end{minipage}%
	\caption{3D numerical solutions $u(x,t)$ (left) and $v(x, t)$ (right) of Problem \eqref{numer1} with $p = 0.1$ }\label{fig3d1}
\end{figure}

\begin{figure}
	\begin{minipage}[h]{0.48\linewidth}
		\centering
		\includegraphics[width=2.46in]{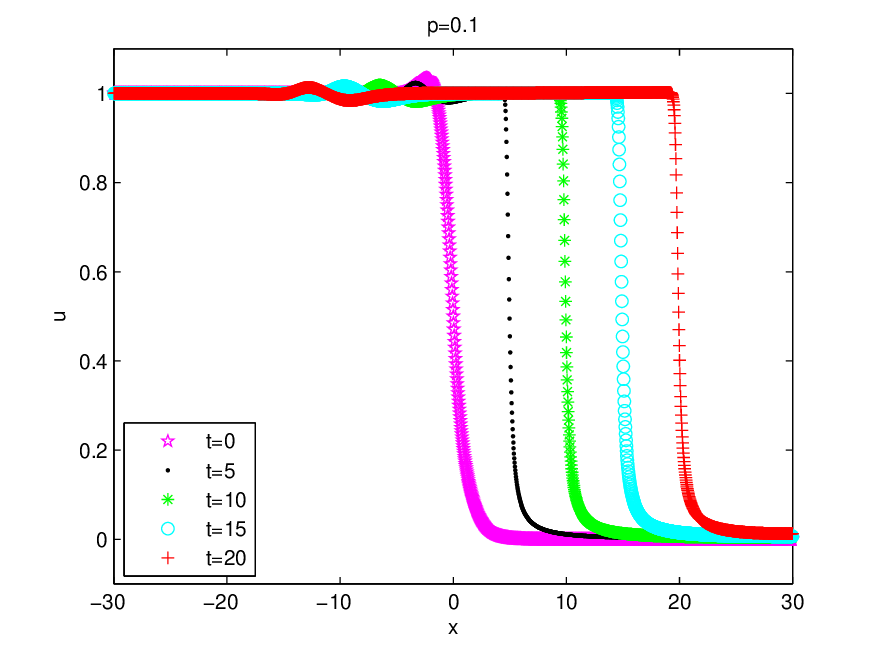}
	\end{minipage}%
	\begin{minipage}[h]{0.48\linewidth}
		\centering
		\includegraphics[width=2.46in]{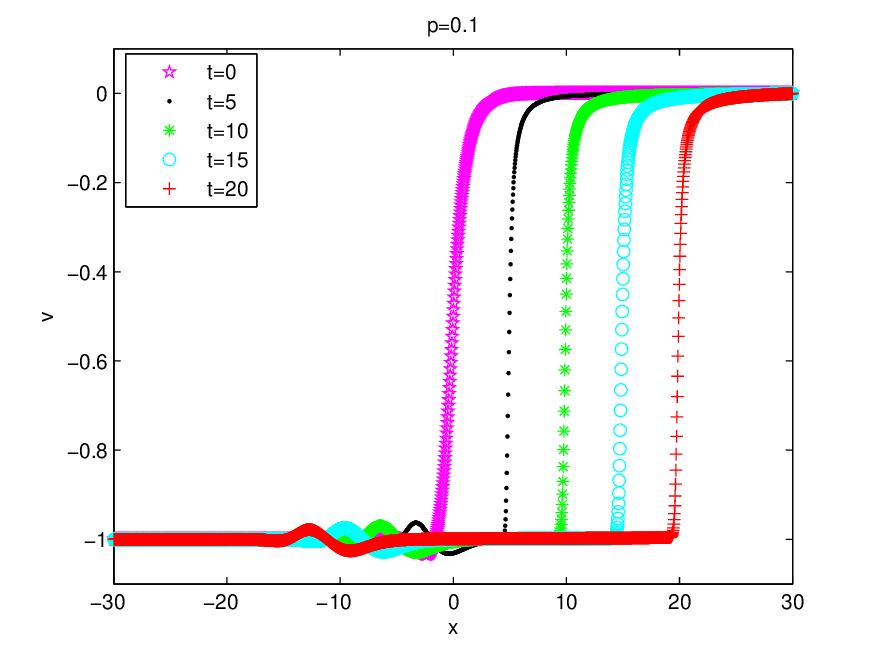}
	\end{minipage}%
	\caption{2D numerical solutions $u(x,t)$ (left) and $v(x, t)$ (right) of Problem \eqref{numer1} with $p = 0.1$ }
\end{figure}

\begin{figure}
	\begin{minipage}[h]{0.48\linewidth}
		\centering
		\includegraphics[width=2.46in]{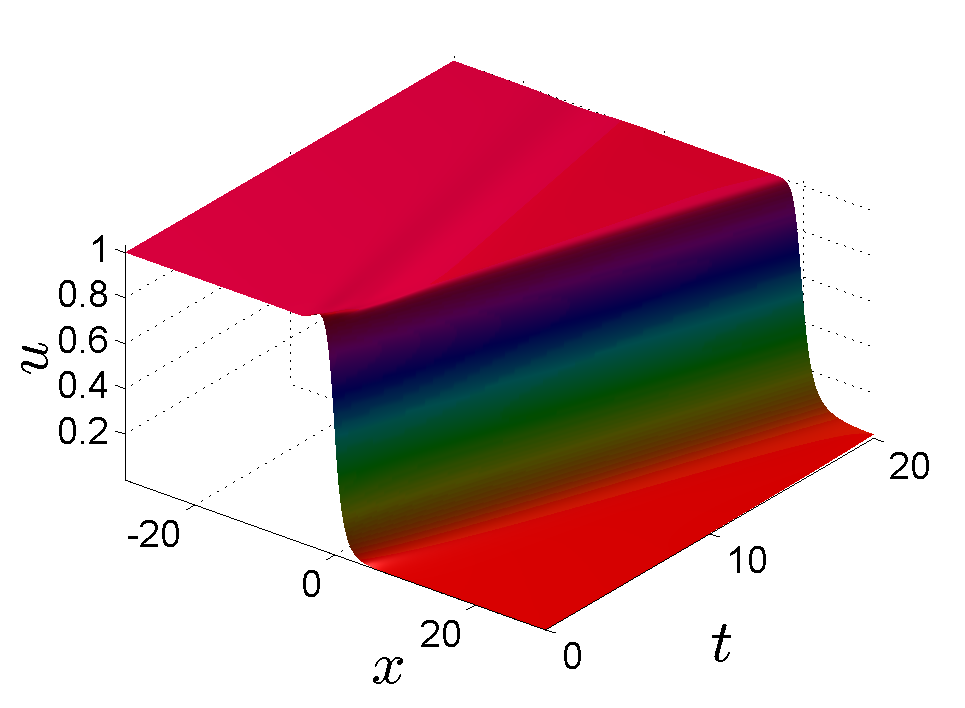}
	\end{minipage}%
	\begin{minipage}[h]{0.48\linewidth}
		\centering
		\includegraphics[width=2.46in]{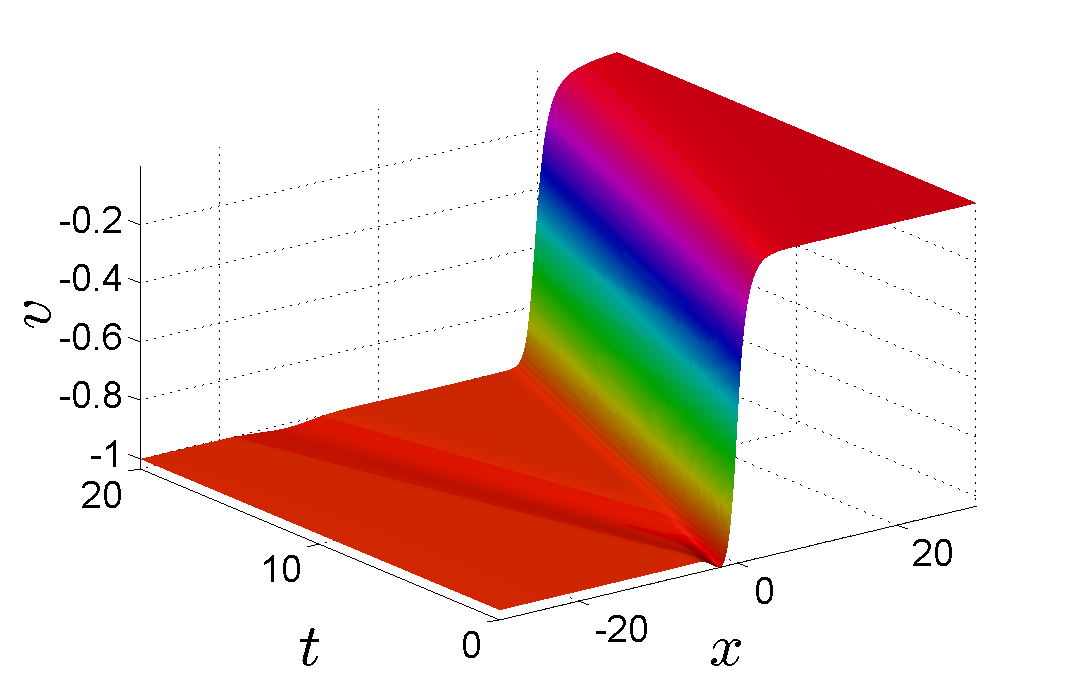}
	\end{minipage}%
	\caption{3D numerical solutions $u(x,t)$ (left) and $v(x, t)$ (right) of Problem \eqref{numer1} with $p = 0.5$ }
\end{figure}

\begin{figure}
	\begin{minipage}[h]{0.48\linewidth}
		\centering
		\includegraphics[width=2.46in]{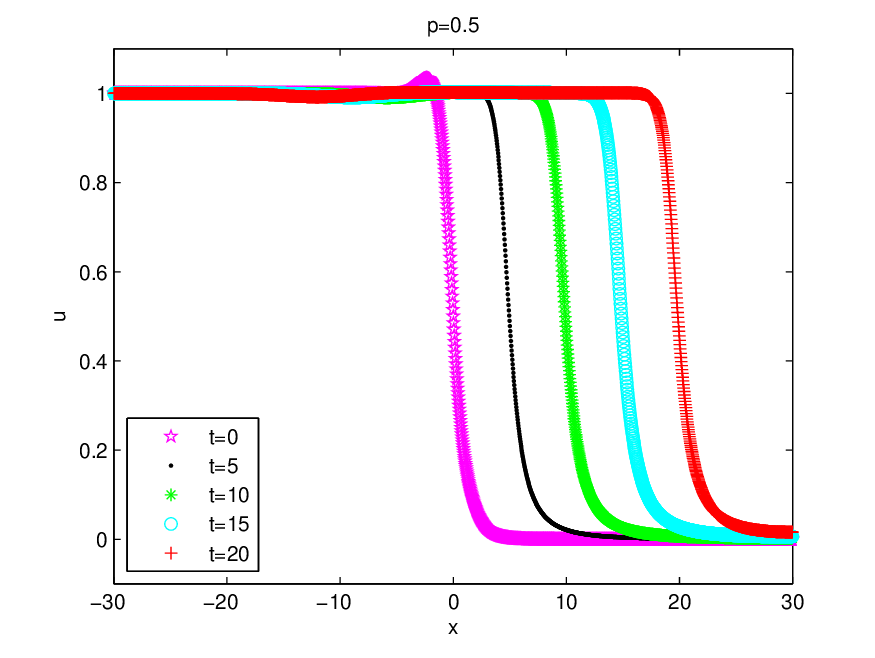}
	\end{minipage}%
	\begin{minipage}[h]{0.48\linewidth}
		\centering
		\includegraphics[width=2.46in]{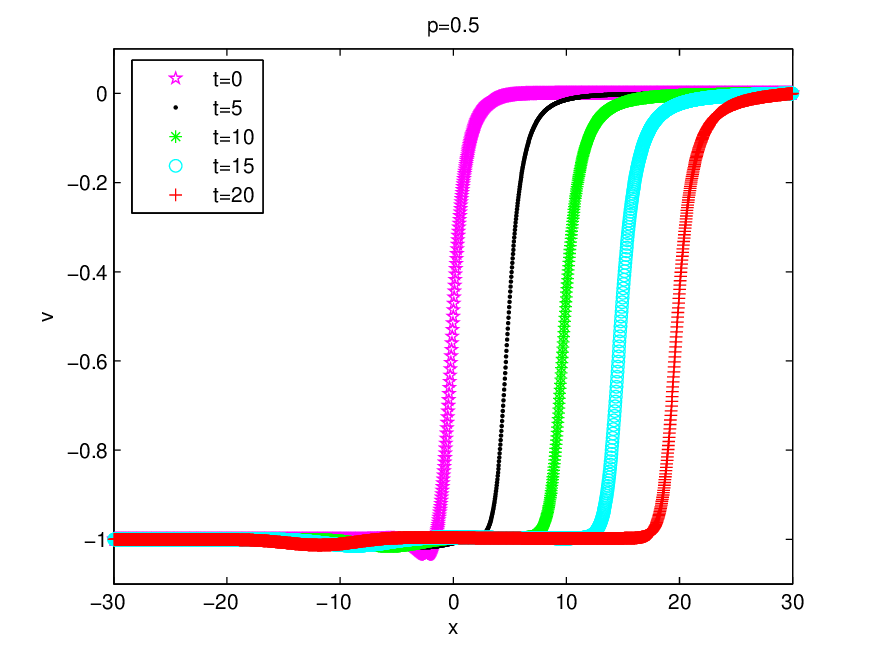}
	\end{minipage}%
	\caption{2D numerical solutions $u(x,t)$ (left) and $v(x, t)$ (right) of Problem \eqref{numer1} with $p = 0.5$ }
\end{figure}

\begin{figure}
	\begin{minipage}[h]{0.48\linewidth}
		\centering
		\includegraphics[width=2.46in]{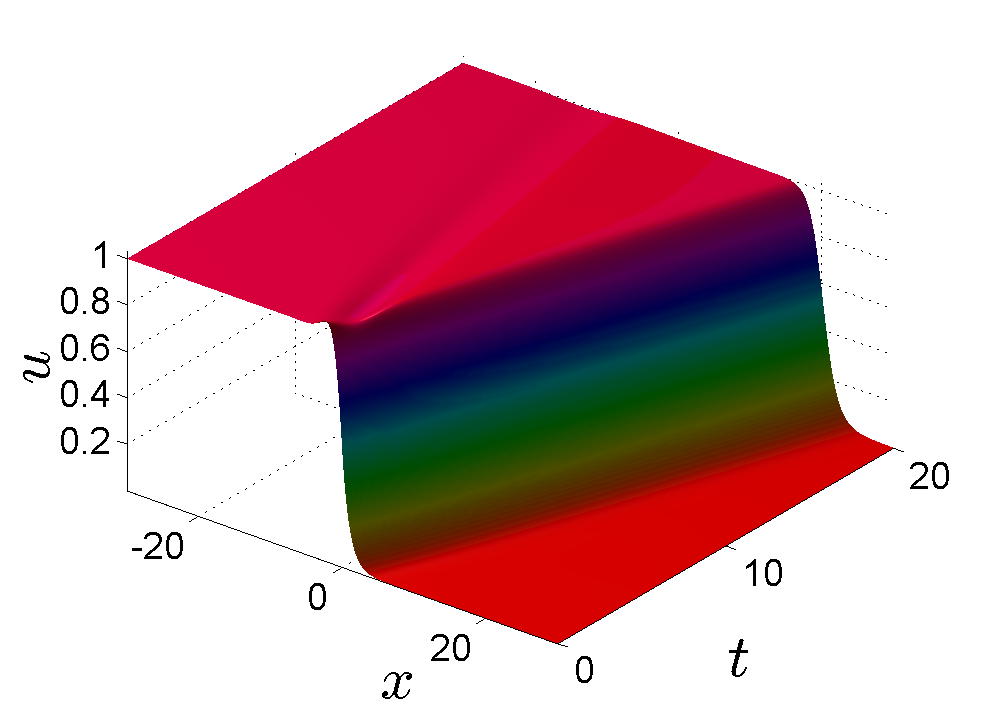}
	\end{minipage}%
	\begin{minipage}[h]{0.48\linewidth}
		\centering
		\includegraphics[width=2.46in]{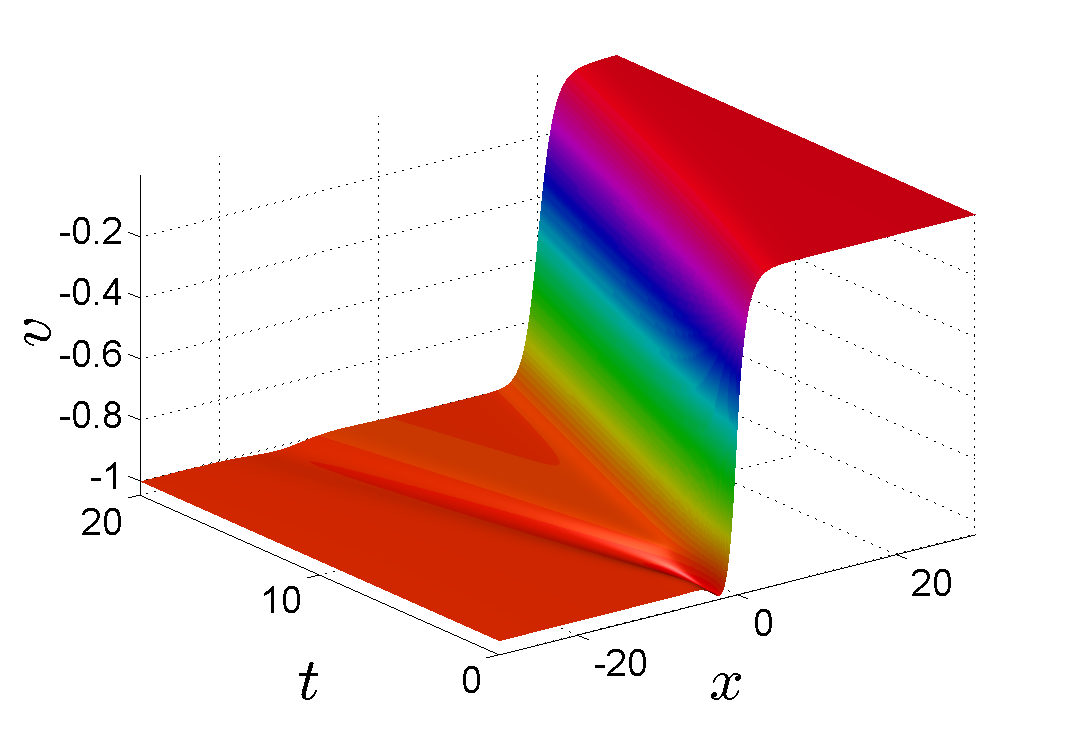}
	\end{minipage}%
	\caption{3D numerical solutions $u(x,t)$ (left) and $v(x, t)$ (right) of Problem \eqref{numer1} with $p = 0.9$ }
\end{figure}

\begin{figure}
	\begin{minipage}[h]{0.48\linewidth}
		\centering
		\includegraphics[width=2.46in]{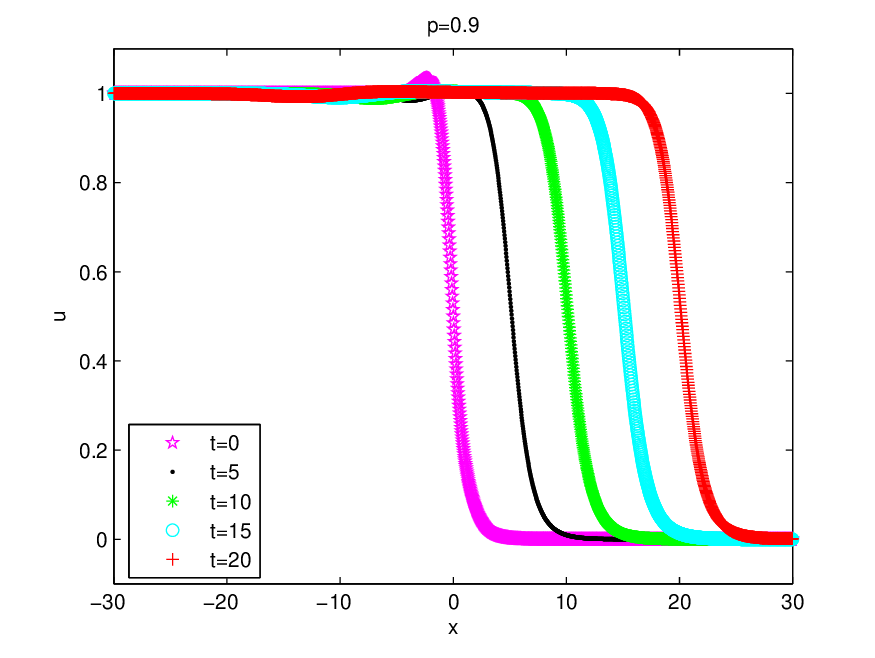}
	\end{minipage}%
	\begin{minipage}[h]{0.48\linewidth}
		\centering
		\includegraphics[width=2.46in]{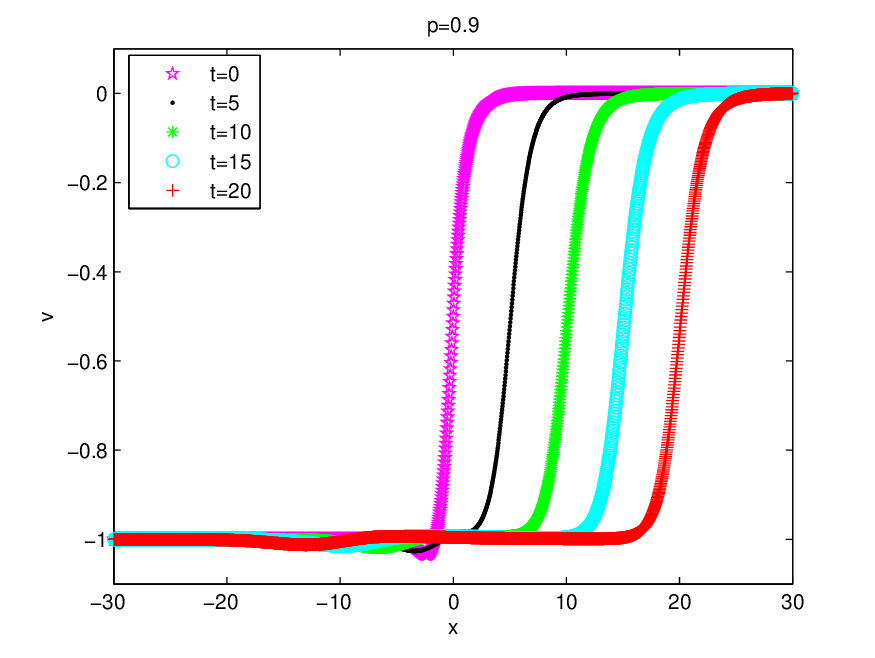}
	\end{minipage}%
	\caption{2D numerical solutions $u(x,t)$ (left) and $v(x, t)$ (right) of Problem \eqref{numer1} with $p = 0.9$ }\label{fig2d3}
\end{figure}

\begin{figure}
	\begin{minipage}[h]{0.48\linewidth}
		\centering
		\includegraphics[width=2.46in]{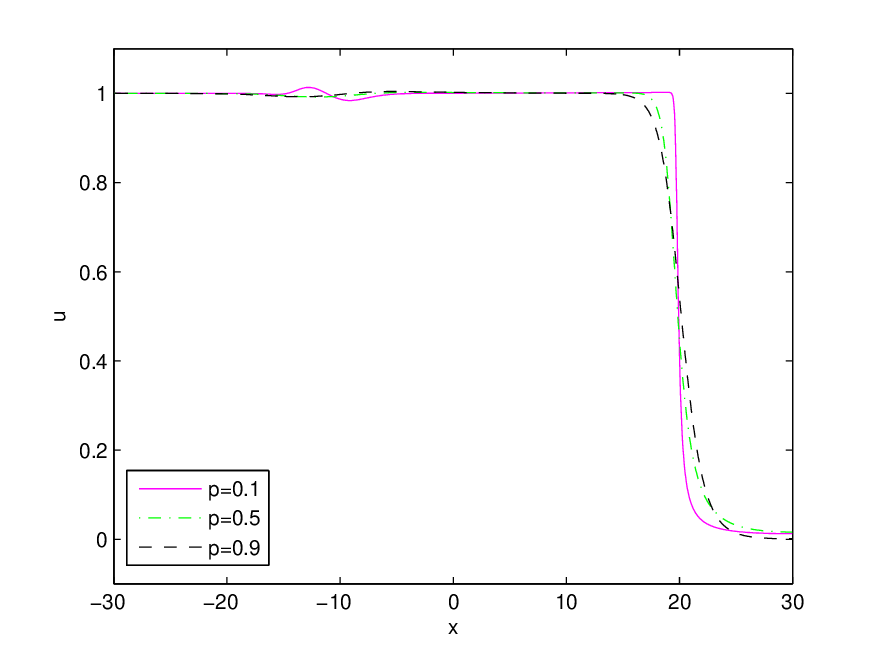}
	\end{minipage}%
	\begin{minipage}[h]{0.48\linewidth}
		\centering
		\includegraphics[width=2.46in]{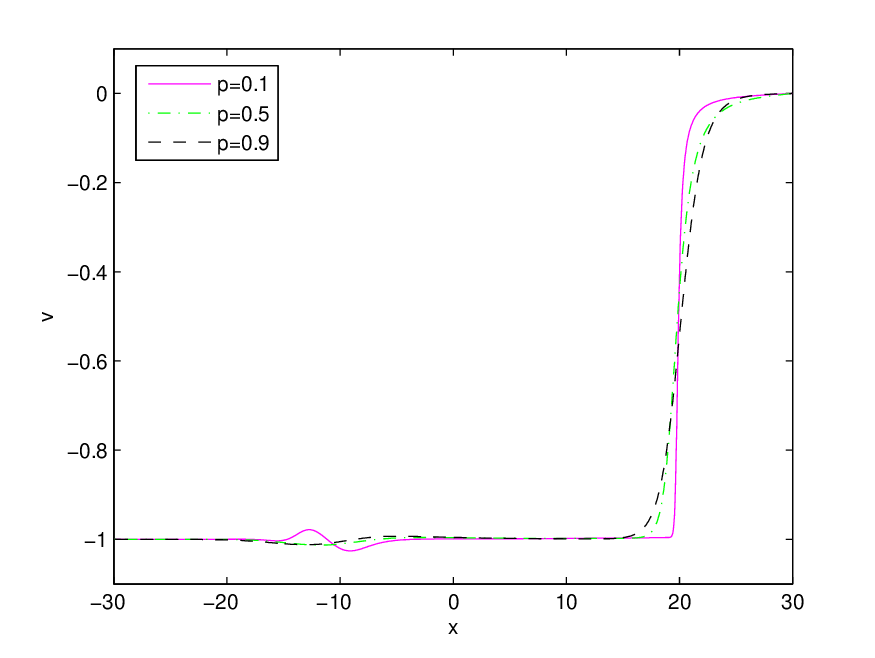}
	\end{minipage}%
	\caption{Numerical solutions $u(x, t)$ and $v(x,t)$ of Problem (5.1) with different $p$ at $t = 20$ }\label{fig2d33}
\end{figure}

To numerically solve the problem conveniently,  we set the computational domain $(x,t)\in(-30,30)\times[0,20]$ and consider problem \eqref{numer1} with homogeneous Neumann boundary conditions.
Then, the spatial discretization is done by the second-order central finite different methods with stepsize $h=0.05$. And the time discretization is done by the Crank-Nicolson method with the stepsize $\tau=0.05$.
The fully-discrete schemes give a system of nonlinear equations, which is solved by Newton iterative method with the iterative residual $\epsilon=10^{-10}$. The numerical results with different parameters $p$ are shown in Figures \ref{fig3d1}-\ref{fig2d3}, respectively. It can be seen from the figures  that the numerical solutions of $u$ and $v$ behave like two traveling waves. The larger the parameter $p$ is, the more obvious dissipative effect can be found. All the numerical results indicate that the traveling wave solutions are stable as time increases. They confirm the theoretical results in this paper. In particular, from Figure \ref{fig2d33}, we see that the stronger fast diffusion (the smaller value of $p$) makes the traveling waves much steeper like shock waves, namely, the effect of fast diffusion with the singular state $u_+=0$ to the structure of traveling waves is essential.

\section{Acknowledgements}
The research of D. Li is supported by the National Natural Science Foundation of China (Nos. 11771162, 12231003). The research of J. Li is supported by the Natural Science Foundation of Jilin Province (No. 20210101144JC) and the National Natural Science Foundation of China (No. 12371216). The research of M. Mei is supported in part by NSERC grant RGPIN 2022-03374.

\end{document}